\pdfoutput=1 
\documentclass[12pt, reqno, oneside]{amsart}

\usepackage[paper=letterpaper, inner=1in, outer=1in, top=1in, bottom=1in]{geometry}

\usepackage{enumitem} 
\setlist{noitemsep, topsep=0ex, labelsep=1em}
\setlist[enumerate, 1]{label=(\Roman*), ref=(\Roman*), left=1em}

\usepackage{booktabs} 
\usepackage{makecell} 
\usepackage{subcaption} 

\usepackage{mlmodern}
\usepackage[T1]{fontenc}
\usepackage{microtype} 

\usepackage[mathscr]{eucal}
\usepackage{amsmath, amssymb, mathtools}

\usepackage[backend=biber, style=alphabetic, maxbibnames=99, maxalphanames=5, url=false, isbn=false, doi=false]{biblatex}
\usepackage{xcolor}
\usepackage{hyperref}
\hypersetup{
	colorlinks=true,
	linkcolor=blue!50!black,
	urlcolor=teal!70!black,
	citecolor=green!50!black,
	breaklinks=true
}
\usepackage[capitalize,nameinlink,noabbrev,nosort]{cleveref}

\newbibmacro{string+doi}[1]{%
	\iffieldundef{doi}{\iffieldundef{url}{#1}{\href{\thefield{url}}{#1}}}{\href{http://dx.doi.org/\thefield{doi}}{#1}}}	
\DeclareFieldFormat[article]{title}{\usebibmacro{string+doi}{\mkbibquote{#1}}}
\DeclareFieldFormat[thesis]{title}{\usebibmacro{string+doi}{\mkbibemph{#1}}}

\usepackage{amsthm, thmtools}

\theoremstyle{plain}
\newtheorem{theorem}{Theorem}[section]
\newtheorem{lemma}[theorem]{Lemma}
\newtheorem{proposition}[theorem]{Proposition}
\newtheorem{corollary}[theorem]{Corollary}

\theoremstyle{definition}
\newtheorem{definition}[theorem]{Definition}
\newtheorem{example}[theorem]{Example}

\theoremstyle{remark}
\newtheorem{remark}[theorem]{Remark}

\usepackage{tikz}
\usetikzlibrary{calc, patterns, arrows, arrows.meta}
\usepackage{ytableau}

\newcommand{\vocab}[1]{\textcolor{blue!50!black}{\emph{#1}}} 

\DeclareSymbolFont{sansops}{T1}{\sfdefault}{m}{n}
\SetSymbolFont{sansops}{bold}{T1}{\sfdefault}{b}{n}
\makeatletter
\renewcommand\operator@font{\mathgroup\symsansops}
\makeatother

\newcommand{\suchthat}{\,:\,}

\newcommand{\rationals}{\mathbb{Q}}

\newcommand{\permact}{\cdot}

\newcommand{\sym}[1]{\mathscr{S}_{#1}} 
\DeclareMathOperator{\adg}{\widehat{dg}}

\DeclareMathOperator{\Arm}{Arm}
\DeclareMathOperator{\leftArm}{Arm\textsuperscript{left}}
\DeclareMathOperator{\rightArm}{Arm\textsuperscript{right}}
\DeclareMathOperator{\leg}{leg}
\DeclareMathOperator{\arm}{arm}
\DeclareMathOperator{\Fill}{F}
\DeclareMathOperator{\NAF}{NAF}
\DeclareMathOperator{\rev}{rev}
\DeclareMathOperator{\inc}{inc}
\DeclareMathOperator{\dec}{dec}
\DeclareMathOperator{\content}{content}
\DeclareMathOperator{\twinv}{twinv}

\newcommand{\swap}[2][]{\mathfrak{t}_{#2}\ifx&#1&\else^{(#1)}\fi}
\newcommand{\sswap}[2][]{\mathfrak{t}_{#2}\ifx&#1&\else^{[#1]}\fi}
\newcommand{\esop}[2][]{\tilde\tau_{#2}\ifx&#1&\else^{(#1)}\fi}
\newcommand{\prob}[1]{\operatorname{prob}_{#1}}

\newcommand{\wt}[1][]{{\operatorname{wt}\ifx&#1&\else^{(#1)}\fi}}
\newcommand{\wtqt}[1][]{{\operatorname{wt}_{q, t}\ifx&#1&\else^{(#1)}\fi}}
\newcommand{\maj}[1][]{{\operatorname{maj}\ifx&#1&\else^{(#1)}\fi}}
\newcommand{\coinv}[1][]{{\operatorname{coinv}\ifx&#1&\else^{(#1)}\fi}}
\newcommand{\inv}[1][]{{\operatorname{inv}\ifx&#1&\else^{(#1)}\fi}}
\newcommand{\Des}[1][]{{\operatorname{Des}\ifx&#1&\else^{(#1)}\fi}}
\newcommand{\des}[1][]{{\operatorname{des}\ifx&#1&\else^{(#1)}\fi}}
\newcommand{\dg}[1][]{{\operatorname{dg}\ifx&#1&\else^{(#1)}\fi}}

\newcommand{\south}[1]{d(#1)}

\newcommand{\algebra}{A}

\newcommand{\atompoly}{\mathcal{A}}
\newcommand{\keypoly}{\mathcal{K}}

\DeclareMathOperator{\id}{id}

\newcommand{\interval}[1]{[#1]}

\usepackage{listofitems}
\newcommand{\composition}[2][]{%
	\def\temp{#2}\ifx\temp\empty%
	\varnothing%
	\else%
	\readlist\thecycle{#2}%
	\foreachitem\i\in\thecycle{\ifnum\icnt=1\else#1\fi\i}%
	\fi%
}

\newcommand{\window}[2][,]{%
	\left[
	\readlist\thecycle{#2}%
	\foreachitem\i\in\thecycle{\ifnum\icnt=1\else#1\fi\i}%
	\right]
}

\newcommand{\ind}[2][]{%
	\readlist\thecycle{#2}%
	\chi_{%
	\foreachitem\i\in\thecycle{\ifnum\icnt=1\else#1\fi\i}%
	}
}

\title[Probabilistic Bijections for Non-Attacking Fillings]{Probabilistic Entry Swapping Bijections for Non-Attacking Fillings}
\author{{Guilherme Zeus} {Dantas e Moura}}
\address{
	Department of Combinatorics \& Optimization\\
	University of Waterloo\\
	Canada
	}
\email{zeus@guilhermezeus.com, zeus.dantasemoura@uwaterloo.ca}
\author{Olya Mandelshtam}
\address{
	Department of Combinatorics \& Optimization\\
	University of Waterloo\\
	Canada
	}
\email{omandels@uwaterloo.ca}
\date{}

\begin{document}

\begin{abstract}
	Non-attacking fillings are combinatorial objects
	central to the theory of Macdonald polynomials.
	A probabilistic bijection for partition-shaped non-attacking fillings
	was introduced by \citeauthor{Man24}~(\citeyear{Man24})
	to prove a compact formula for symmetric Macdonald polynomials.

	In this work,
	we generalize this probabilistic bijection
	to composition-shaped non-attacking fillings.
	As an application,
	we give a bijective proof to extend a symmetry theorem
	for permuted-basement Macdonald polynomials
	established by \citeauthor{Ale19}~(\citeyear{Ale19}),
	proving a version with fewer assumptions.
\end{abstract}

\maketitle

\section{Introduction}

The \emph{symmetric Macdonald polynomials}
\(P_\lambda(x_1, \ldots, x_n; q, t)\),
indexed by partitions \(\lambda\),
form a family of symmetric polynomials
in the variables \(x_1, \ldots, x_n\) over the field \(\rationals(q, t)\).
They are the unique monic basis of
the ring of symmetric polynomials of \(\rationals(q, t)[x_1, \ldots, x_n]\)
satisfying certain orthogonality and triangularity conditions.

The \emph{nonsymmetric Macdonald polynomials}
\(E_\alpha(x_1, \ldots, x_n; q, t)\),
indexed by weak compositions \(\alpha=\composition{\alpha_1,\ldots,\alpha_n}\),
were introduced in \cite{Opd95}, \cite{Mac96}, and \cite{Che95}
as a generalization of the symmetric Macdonald polynomials.
They are the monic eigenfunctions of the Cherednik--Dunkl operators
and form a basis of the polynomial ring \(\rationals(q, t)[x_1, \ldots, x_n]\).

The \emph{permuted-basement Macdonald polynomials}
\(E_\alpha^\sigma(x_1, \ldots, x_n; q, t)\),
first defined by \cite{Fer11},
introduce an additional level of indexing,
pairing a weak composition \(\alpha\) with a permutation \(\sigma\in S_n\).
They are obtained from the nonsymmetric Macdonald polynomials
by applying sequences of Demazure--Lusztig operators
(see \cref{sec:permuted basement def}).
For any fixed \(\sigma\),
the set of permuted-basement Macdonald polynomials
\(E_\alpha^\sigma(x_1,\ldots,x_n;q,t)\)
ranging over all weak compositions \(\alpha\)
forms a basis for \(\rationals(q, t)[x_1, \ldots, x_n]\).
Moreover, for a partition \(\lambda=(\lambda_1\geq\cdots\geq\lambda_n\geq 0)\),
it was observed in \cite{CMW22} that
\begin{equation}
	\label{eq:P in Es}
	P_{\lambda}(x_1,\ldots,x_n;q,t)
	=
	\sum_{\mu\in S_n\permact \lambda}
	E_{\inc(\lambda)}^{\sigma_\mu}(x_1,\ldots,x_n;q,t),
\end{equation}
where \(S_n\permact \lambda\) is the set of rearrangements of \(\lambda\),
\(\inc(\lambda)=(\lambda_n,\ldots,\lambda_1)\),
and \(\sigma_\mu\) is the longest permutation
that rearranges \(\inc(\lambda)\) into \(\mu\).

There is a combinatorial interpretation
for permuted-basement nonsymmetric Macdonald polynomials \(E_\alpha^\sigma\)
due to \cite{Fer11}
as the generating function of non-attacking fillings
of composition shape \(\alpha\) and \emph{basement} \(\sigma\).
This interpretation generalizes the tableaux formula of \citeauthor{HHL08}
for \(P_{\lambda}\) \cite{HHL08} (see \cref{sec:tableau-formula}).

\textcite[Theorem~22]{Ale19} established a partial symmetry result
for permuted-basement Macdonald polynomials under certain conditions.
Using \cite[Lemma~20 and Proposition~15]{Ale19},
we reformulate the result as follows.

\begin{restatable}[equivalent to {\cite[Theorem~22]{Ale19}}]{theorem}{symmetryAle}
	\label{theorem:symmetry-Ale19}
	Let \(\alpha=\composition{\alpha_1,\ldots,\alpha_n}\) be a composition
	with \(\alpha_i=\alpha_{i+1}\) for some \(i\in\interval{n-1}\).
	Let \(\sigma \in \sym{n}\), such that \(\sigma_{i+1} = \sigma_i \pm 1\).
	Then,
	\begin{equation*}
		E_{\alpha}^{\sigma}(\mathbf{x}; q, t)
		= E_{\alpha}^{\sigma s_i}(\mathbf{x}; q, t).
	\end{equation*}
\end{restatable}

Our main result is stated in the following theorem,
which extends \cref{theorem:symmetry-Ale19}
by removing the assumption that \(\sigma_i\) and \(\sigma_{i+1}\) are consecutive.

\begin{restatable}{theorem}{main}
	\label{theorem:main}
	Let \(\alpha=\composition{\alpha_1,\ldots,\alpha_n}\) be a composition
	with \(\alpha_i=\alpha_{i+1}\) for some \(i\in\interval{n-1}\),
	and let \(\sigma \in \sym{n}\).
	Then,
	\begin{equation*}
		E_{\alpha}^{\sigma}(\mathbf{x}; q, t)
		= E_{\alpha}^{\sigma s_i}(\mathbf{x}; q, t).
	\end{equation*}
\end{restatable}

\cref{theorem:main} is proved algebraically
from \cref{theorem:symmetry-Ale19} in \cref{sec:proof-alternate}.
However, in this article,
our focus is on providing a fully combinatorial proof
of \cref{theorem:symmetry-Ale19,theorem:main}
by constructing a probabilistic bijection on non-attacking fillings.

The combinatorial interpretation of \cref{theorem:main} is that
swapping the basement entries in two adjacent columns of the same height
does not change the generating function for non-attacking fillings.
Our approach for proving this builds upon a (classical) bijection
introduced in \cite{CHMMW22} for unrestricted fillings of partition shape,
in which two adjacent entries of the filling
in the bottom row of columns of the same height are swapped
while preserving the overall weight of the filling.
In \cite{Man24}, the second author extended this technique
by introducing a probabilistic bijection
for partition-shaped non-attacking fillings.
In this article,
we generalize this probabilistic bijection
to composition-shaped non-attacking fillings,
thereby giving a bijective proof of \cref{theorem:main}.

As implications of \cref{theorem:main},
certain assumptions can be removed from
theorems in \cite{Ale19} and \cite{CMW22},
which we summarize in \cref{sec:applications}.

This article is organized as follows.
\cref{sec:definitions} contains preliminaries on
permuted-basement Macdonald polynomials
and their tableaux formula in terms of non-attacking fillings.
\cref{sec:probabilistic} defines a probabilistic bijection on non-attacking fillings,
leading to a combinatorial proof of \cref{theorem:main} in \cref{sec:proof}.
Finally, \cref{sec:applications} discusses some implications of \cref{theorem:main}.

\subsection*{Acknowledgements}
We would like to thank
Per Alexandersson,
Sarah Mason,
and Arun Ram
for helpful discussions.
OM was supported by NSERC grant RGPIN-2021-02568.

\section{Preliminaries}
\label{sec:definitions}

\subsection{Permutations and compositions}

Throughout the paper, let \(n\) be a nonnegative integer.
The \vocab{symmetric group} \(\sym{n}\) is
the group of all permutations of the set
\(\interval{n} = \{1, 2, \ldots, n\}\).
We write permutations in one-line notation
\(\sigma = \window{\sigma_1, \sigma_2, \ldots, \sigma_n}\),
with brackets to distinguish them from compositions.
Given two permutations \(\sigma, \pi \in \sym{n}\),
their product is the permutation \(\sigma \pi \in \sym{n}\)
defined by \((\sigma \pi)_i = \sigma_{\pi_i}\).
The \vocab{reverse} \(\rev(\sigma)\) of
a permutation \(\sigma \in \sym{n}\) is the permutation
\(\sigma w_0 = \window{\sigma_n, \sigma_{n-1}, \dots, \sigma_1}\),
where \(w_0 = \window{n, n-1, \cdots, 1}\).
The symmetric group \(\sym{n}\) is generated by
the \vocab{simple transpositions} \(s_1, s_2, \ldots, s_{n-1}\),
where \(s_i\) swaps \(i\) and \(i+1\).
The \vocab{length} \(\ell(\sigma)\) of a permutation \(\sigma \in \sym{n}\) is
the smallest number of simple transpositions whose product equals \(\sigma\).
A \vocab{reduced expression} of a permutation \(\sigma\) is
a sequence \(s_{i_1},s_{i_2},\ldots,s_{i_{\ell(\sigma)}}\) of simple transpositions
whose product \(s_{i_1}s_{i_2}\cdots s_{i_{\ell(\sigma)}}\) equals \(\sigma\).

A \vocab{composition of length \(n\)} is a sequence
\(\alpha = \composition{\alpha_1, \alpha_2, \ldots, \alpha_n}\)
of nonnegative integers.
A \vocab{partition} is a weakly decreasing composition,
and an \vocab{antipartition} is a weakly increasing composition.
Given a composition \(\alpha\), let \(\dec(\alpha)\) and \(\inc(\alpha)\) denote
the partition and antipartition obtained by sorting the parts of \(\alpha\)
in weakly decreasing and weakly increasing order, respectively.
Define \(\rev(\alpha)=\composition{\alpha_n, \alpha_{n-1}, \ldots, \alpha_1}\)
to be the composition given by reversing the parts of \(\alpha\).

The \vocab{exponent notation} for a composition \(\alpha = \composition{\alpha_1, \ldots, \alpha_n}\)
is given by a sequence of nonnegative integers \(p_1, p_2, \ldots, p_k\)
and positive integers \(m_1, m_2, \ldots, m_k\) that sum to \(n\)
such that \(p_i \neq p_{i+1}\) for \(i\in[k-1]\), written as
\begin{equation*}
	\alpha = \composition{
		\underbrace{p_1 p_1 \ldots p_1}_{m_1},
		\underbrace{p_2 p_2 \ldots p_2}_{m_2},
		\ldots,
		\underbrace{p_k p_k \ldots p_k}_{m_k}
	} = p_1^{m_1} p_2^{m_2} \ldots p_k^{m_k}.
\end{equation*}

Given a permutation \(\pi \in \sym{n}\) and a composition \(\alpha = \composition{\alpha_1, \ldots, \alpha_n}\),
the \vocab{left action} of \(\pi\) on \(\alpha\) is the composition
\(\pi \permact \alpha = \composition{\alpha_{\pi^{-1}(1)}, \alpha_{\pi^{-1}(2)}, \ldots, \alpha_{\pi^{-1}(n)}}\).
In particular,
this choice of notation guarantees that
\(\sigma \permact (\pi \permact \alpha) = (\sigma \pi) \permact \alpha\)
for all \(\sigma, \pi \in \sym{n}\) and
\(\alpha=\composition{\alpha_1,\ldots,\alpha_n}\).
Note that \(\rev(\alpha) = w_0 \permact \alpha\).

\begin{example}
	For example, \(\alpha=\composition{3, 1, 0, 2, 2, 0}\) is
	a composition of length \(6\) with exponent notation \(3^11^10^12^20^1\).
	It rearranges to the partition \(\dec(\alpha) = \composition{3, 2, 2, 1, 0, 0}\),
	the antipartition \(\inc(\alpha) = \composition{0, 0, 1, 2, 2, 3}\),
	and \(\rev(\alpha)=\composition{0,2,1,0,2,3}\).
	The permutation \(\pi = \window{3, 1, 2, 4, 6, 5}\) has reduced expression \(s_2s_1s_5\).
	The result of the left action of \(\pi\) on \(\alpha\) is
	the composition \(\pi \permact \alpha = \composition{1, 0, 3, 2, 0, 2}\),
	with a graphical representation shown in \Cref{fig:left-action-on-composition}.
\end{example}

\begin{figure}[htbp]
	\begin{tikzpicture}[x=.4cm, y=-.5cm]
		\node[anchor=east] at (0.7, 0) {\(\alpha = \)};
		\foreach \i [count = \j] in {3, 1, 0, 2, 2, 0}{
				\node (A\j) at (\j, 0) {\(\i\)};
			}
		\node[anchor=east] at (0.7, 2) {\(\pi \permact \alpha = \)};
		\foreach \i [count = \j] in {1, 0, 3, 2, 0, 2}{
				\node (B\j) at (\j, 2) {\(\i\)};
			}
		\node[anchor=east, color=red!70!black] at (7, 1) {\(\scriptstyle \pi\)};
		\foreach \i [count = \j] in {3, 1, 2, 4, 6, 5}{
				\draw[->, color=red!70!black] (A\j.south) -- (B\i.north);
			}
	\end{tikzpicture}
	\caption{
		The left action of \(\pi = \window{3, 1, 2, 4, 6, 5}\)
		on \(\alpha = \composition{3, 1, 0, 2, 2, 0}\).
		}
	\label{fig:left-action-on-composition}
\end{figure}

\begin{remark}
	Given permutations \(\sigma, \pi \in \sym{n}\) and
	a composition \(\alpha=\composition{\alpha_1,\ldots,\alpha_n}\),
	\cite{Ale19} writes \(\pi \alpha\) to mean \(\pi \permact \alpha\),
	and writes \(\pi \sigma\) to mean \(\pi \permact \sigma\),
	interpreting \(\sigma\) as a composition,
	which we write as a product \(\pi \sigma^{-1}\).
\end{remark}

\subsection{Permuted-basement Macdonald polynomials}
\label{sec:permuted basement def}

We follow the exposition of \cite{Fer11} and the notation of \cite{Ale19}
for the definition of nonsymmetric Macdonald polynomials
and permuted-basement Macdonald polynomials.
For more background on this topic, see \cite{GR22},
where these polynomials are referred to as the \emph{relative Macdonald polynomials}.

Let \(x_1\), \dots, \(x_n\) be indeterminates, denoted collectively as \(\mathbf{x}\).
Let \(q, t\) be parameters.
We define nonsymmetric Macdonald polynomials and permuted-basement Macdonald polynomials,
which are polynomials in the ring \(\rationals(q, t)[\mathbf{x}]\).

A permutation \(\sigma \in \sym{n}\) acts on
a polynomial \(f(\mathbf{x})\) by permuting the variables,
that is,
\(\sigma (\mathbf{x}^{\alpha}) = \mathbf{x}^{\sigma \permact \alpha}\)
for a composition \(\alpha\).

The \vocab{Demazure--Lusztig operators}
\(T_1\), \dots, \(T_{n-1}\) on \(\rationals(q, t)[\mathbf{x}]\)
are defined by
\begin{equation*}
	T_i(f) = ts_i(f) + (1-t) x_{i+1} \frac{f - s_i(f)}{x_i - x_{i+1}},
\end{equation*}
for \(i \in \interval{n-1}\),
and the operator \(g\) on \(\rationals(q, t)[\mathbf{x}]\) is defined by
\begin{equation*}
	g(f(x_1, \ldots, x_n)) = f(x_2, x_3, \ldots, x_n, q^{-1}x_1).
\end{equation*}
For a reduced expression \( \sigma = s_{i_1} \cdots s_{i_k} \),
we write \(T_\sigma\) to denote the composition \(T_{i_1} \circ \cdots \circ T_{i_k}\).

The \vocab{Cherednik--Dunkl operators}
\(Y_1\), \dots, \(Y_n\) on \(\rationals(q, t)[\mathbf{x}]\)
are defined by,
for each \(i \in \interval{n}\),
\begin{equation*}
	Y_i = t^{i-1} T_{i-1}^{-1} \cdots T_1^{-1} g T_{n-1} \cdots T_i.
\end{equation*}

\begin{definition}[{\cite[Theorem 4.1.4]{Fer11}}]
	For a composition \(\alpha=\composition{\alpha_1,\ldots,\alpha_n}\),
	the \vocab{nonsymmetric Macdonald polynomial} \(E_\alpha=E_\alpha(\mathbf{x}; q, t)\) is
	the unique element of \(\rationals(q, t)[\mathbf{x}]\) such that
	the coefficient of \(\mathbf{x}^{\alpha}\) is \(1\),
	and for all \(i\in\interval{n}\),
	\begin{equation*}
		Y_i E_\alpha = q^{-\alpha_i} t^{k_i} E_\alpha,
	\end{equation*}
	where
	\(
		k_i = |\{ j \in \interval{i-1} \suchthat \alpha_j > \alpha_i \}|
			+ |\{ j \in \interval{i+1, n} \suchthat \alpha_j \geq \alpha_i \}|
	\).
\end{definition}

\begin{remark}
	We note that there is a typo in \cite{Fer11}:
	the corresponding formula for \(k_i\) has a minus sign instead of a plus sign.
	Our formulation aligns with the definition in \cite{Kno97},
	although the notation there differs.
\end{remark}

\begin{example}
	\label{example:nonsymmetric-macdonald-polynomial}
	Let \(n = 4\) and \(\alpha = \composition{1, 0, 1, 1}\).
	Then, \(k_1 = 2\), \(k_2 = 3\), \(k_3 = 1\) and \(k_4 = 0\).
	Therefore, the nonsymmetric Macdonald polynomial
	\(E_{\composition{1, 0, 1, 1}}(\mathbf{x}; q, t)\)
	is the unique element of \(\rationals(q, t)[x_1, x_2, x_3, x_4]\) such that
	the coefficient of \(x_1 x_3 x_4\) is \(1\), and
	\begin{align*}
		Y_1 E_{\composition{1, 0, 1, 1}} & = q^{-1} t^{2} E_{\composition{1, 0, 1, 1}}, &
		Y_2 E_{\composition{1, 0, 1, 1}} & = q^{0}  t^{3} E_{\composition{1, 0, 1, 1}}, \\
		Y_3 E_{\composition{1, 0, 1, 1}} & = q^{-1} t^{1} E_{\composition{1, 0, 1, 1}}, &
		Y_4 E_{\composition{1, 0, 1, 1}} & = q^{-1} t^{0} E_{\composition{1, 0, 1, 1}}.
	\end{align*}
	The unique polynomial satisfying the above is
	\begin{equation*}
		E_{\composition{1, 0, 1, 1}}(x_1, x_2, x_3, x_4; q, t)
		=
		x_1 x_2 x_3 \frac{1 - t}{1 - q t^2} +
		x_1 x_2 x_4 \frac{1 - t}{1 - q t^2} +
		x_1 x_3 x_4.
	\end{equation*}
\end{example}

The following definition is a reformulation based on \cite[Corollary 16]{Ale19} and \cite[Remark 4.4.5]{Fer11}.

\begin{definition}[Permuted-basement Macdonald polynomials]
	\label{definition:permutedbasement}
	For a composition \(\alpha=\composition{\alpha_1,\ldots,\alpha_n}\)
	and a permutation \(\sigma \in \sym{n}\),
	the \vocab{permuted-basement Macdonald polynomial}
	\(E_\alpha^\sigma = E_{\alpha}^{\sigma}(\mathbf{x}; q, t)\)
	is given by
	\begin{equation*}
		E_{\alpha}^{\sigma} =
		t^{-\twinv(\alpha, \sigma)} T_{\rev(\sigma)} E_{\rev(\alpha)},
	\end{equation*}
	where
	\(
		\twinv(\alpha, \sigma) =
		\left| \{
			(i, j) \suchthat
			i < j \text{ and }
			\alpha_i \geq \alpha_j \text{ and }
			\sigma_i < \sigma_j
		\} \right|
	\).
\end{definition}

Note that, for \(\sigma = w_0 = \window{n, n-1, \ldots, 1}\),
the permuted-basement Macdonald polynomial \(E_{\alpha}^{w_0}(\mathbf{x}; q, t)\)
is the nonsymmetric Macdonald polynomial \(E_{\rev(\alpha)}(\mathbf{x}; q, t)\).

\begin{remark}
	We note a typo in \cite{Ale19} concerning the definition of \(\twinv(\alpha, \sigma)\):
	the inequality involving \(\sigma\) is reversed.
	This error becomes apparent by considering the special case \(E_\alpha^{w_0}\).
	Our formulation aligns with the definition provided in \cite{Fer11},
	although we use different notation.
\end{remark}

\begin{remark}
	\label{remark:notation-permuted-basement-macdonald-polynomial-Ale19}
	Our notation for the permuted-basement Macdonald polynomial agrees with
	the notation in \cite{Ale19,CMW22},
	which is different from the notation in \cite{Fer11}.
	Notably, the permuted-basement Macdonald polynomial \(E_{\alpha}^{\sigma}(\mathbf{x}; q, t)\)
	is denoted by \(E_{\rev(\alpha)}^{\rev(\sigma)}(\mathbf{x};q,t)\) in \cite{Fer11}.
\end{remark}

\begin{example}
	\label{example:permuted-basement-macdonald-polynomial}
	Let \(n = 4\),
	\(\alpha = \composition{1, 1, 0, 1}\),
	and \(\sigma = \window{2, 4, 1, 3}\).
	Then, \(\rev(\sigma)=\window{3, 1, 4, 2}\) has reduced expression \(s_2 s_1 s_3\),
	so the Demazure--Lusztig operator \(T_{\rev(\sigma)}\) is given by \(T_2 T_1 T_3\).
	We compute
	\(
		\twinv(\alpha, \sigma) =
		\twinv(\composition{1, 1, 0, 1}, \window{2, 4, 1, 3}) =
		|\{(1, 2), (1, 4)\}| = 2
	\).
	Therefore,
	the permuted-basement Macdonald polynomial
	\(E_{\composition{1, 1, 0, 1}}^{\window{2, 4, 1, 3}}(\mathbf{x}; q, t)\)
	is given by
	\begin{equation*}
		E_{\composition{1, 1, 0, 1}}^{\window{2, 4, 1, 3}}(\mathbf{x}; q, t)
		=
		t^{-2}
		T_2 T_1 T_3
		E_{\composition{1, 0, 1, 1}}(\mathbf{x}; q, t),
	\end{equation*}
	which can be computed from \cref{example:nonsymmetric-macdonald-polynomial} as
	\begin{equation*}
		E_{\composition{1, 1, 0, 1}}^{\window{2, 4, 1, 3}}
		=
		x_{1} x_{2} x_{3} \frac{t(1 - t)}{1 - q t^2} +
		x_{1} x_{3} x_{4} \frac{1 - t}{1 - q t^2} +
		x_{2} x_{3} x_{4}.
	\end{equation*}
\end{example}

\begin{example}
	Let \(n = 4\),
	\(\alpha = \composition{1, 1, 0, 1}\),
	and \(\tau = \window{4, 2, 1, 3}\).
	Then, \(\rev(\tau) = \window{3, 1, 4, 2}\) has a reduced expression \(\rev(\tau) = s_2 s_1\),
	so the Demazure--Lusztig operator \(T_{\rev(\tau)}\) is given by \(T_2 T_1\).
	We compute
	\(
		\twinv(\alpha, \tau) =
		\twinv(\composition{1, 1, 0, 1}, \window{4, 2, 1, 3}) =
		|\{(1, 4)\}| = 1
	\).
	Therefore, the permuted-basement Macdonald polynomial
	\(E_{\composition{1, 1, 0, 1}}^{\window{4, 2, 1, 3}}(\mathbf{x}; q, t)\)
	is given by
	\begin{equation*}
		E_{\composition{1, 1, 0, 1}}^{\window{4, 2, 1, 3}}(\mathbf{x}; q, t)
		=
		t^{-1}
		T_2 T_1
		E_{\composition{1, 0, 1, 1}}(\mathbf{x}; q, t),
	\end{equation*}
	which can be computed as
	\begin{equation*}
		E_{\composition{1, 1, 0, 1}}^{\window{4, 2, 1, 3}}
		=
		x_{1} x_{2} x_{3} \frac{t(1 - t)}{1 - q t^2} +
		x_{1} x_{3} x_{4} \frac{1 - t}{1 - q t^2} +
		x_{2} x_{3} x_{4}.
	\end{equation*}
\end{example}

\subsection{
	Symmetry theorem for
	\texorpdfstring{\(E_\alpha^\sigma\)}{permuted-basement Macdonald polynomials}
}
\label{sec:proof-alternate}

In this section,
we establish that \cref{theorem:main} is a direct corollary
of \cref{definition:permutedbasement,theorem:symmetry-Ale19}.
This connection is not immediately apparent from
the original statement of the result in \cite[Theorem~22]{Ale19}.

\symmetryAle*

\cref{theorem:symmetry-Ale19} states that
swapping the basement entries in two adjacent columns of the same height
does not change the polynomial,
provided the entries are consecutive.
This corresponds to the permutation
\(\window{\sigma_1, \cdots, \sigma_{i-1}, \sigma_{i+1}, \sigma_i, \sigma_{i+2}, \cdots, \sigma_n} = \sigma s_i\).
Note that \cite{Ale19} writes this permutation differently,
following different conventions for permutation products and
focusing on permuting entries rather than positions.

\main*

\begin{proof}[Proof of \cref{theorem:main} as a corollary of \cref{theorem:symmetry-Ale19}]
	Since \(\alpha_i=\alpha_{i+1}\) and \((w_0)_{i+1} = (w_0)_i - 1\),
	applying \cref{theorem:symmetry-Ale19} to the permutation \(w_0\)
	gives \(E_\alpha^{w_0} = E_\alpha^{w_0 s_i}\).
	Using \cref{definition:permutedbasement},
	\begin{equation*}
		E_\alpha^{w_0} =
		E_\alpha^{w_0 s_i} =
		t^{-\twinv(\alpha, w_0 s_i)} T_{\rev(w_0 s_i)} E_\alpha^{w_0} =
		t^{-1} T_{n-i} E_\alpha^{w_0}.
	\end{equation*}

	Without loss of generality, assume \(\sigma_{i+1} = \sigma_i - 1\).
	Then, \(\ell(\sigma) > \ell(\sigma s_i)\).
	For ease of reading, let \(\tau=\rev(\sigma)\),
	so that \(\rev(\sigma s_i)=\tau s_{n-i}\).
	\cref{definition:permutedbasement} implies that
	\begin{equation*}
		E_\alpha^\sigma = t^{-\twinv(\alpha, \sigma)} T_{\tau} E_\alpha^{w_0}
		\qquad \text{and} \qquad
		E_\alpha^{\sigma s_i} = t^{-\twinv(\alpha, \sigma s_i)} T_{\tau s_{n-i}} E_\alpha^{w_0}.
	\end{equation*}
	Let \(s_{i_1}\cdots s_{i_k}\) be a reduced expression for \(\tau\).
	Then \(s_{i_1}\cdots s_{i_k}s_{n-i}\) is a reduced expression for \(\tau s_{n-i}\),
	since \(\ell(\tau) < \ell(\tau s_{n-i})\),
	and thus \(T_{\tau s_{n-i}} = T_{\tau} \circ T_{n-i}\).
	From the definition of \(\twinv(\alpha, \sigma)\), we check that
	\( \twinv(\alpha, \sigma s_i) = \twinv(\alpha, \sigma) + 1 \),
	and consequently, we compute
	\begin{equation*}
		E_\alpha^{\sigma s_i}
		= t^{-\twinv(\alpha, \sigma s_i)} T_{\tau} T_{n-i} E_\alpha^{w_0}
		= t^{-\twinv(\alpha, \sigma)} T_{\tau} \left( t^{-1} T_{n-i} E_\alpha^{w_0} \right)
		= t^{-\twinv(\alpha, \sigma)} T_{\tau} E_\alpha^{w_0}
		= E_\alpha^\sigma,
	\end{equation*}
	as desired.
\end{proof}

\subsection{Non-attacking fillings}

The \vocab{skyline diagram} of a composition
\(\alpha = \composition{\alpha_1, \ldots, \alpha_n}\)
is the set
\begin{equation*}
	\dg(\alpha) = \left\{ (i, r) \suchthat i \in \interval{n} \text{ and } r \in \interval{\alpha_i} \right\},
\end{equation*}
where the first coordinate indexes the columns and the second coordinate indexes the rows.
The \vocab{augmented skyline diagram} of a composition
\(\alpha = \composition{\alpha_1, \ldots, \alpha_n}\)
is the set
\begin{equation*}
	\adg(\alpha) = \dg(\alpha) \cup \left\{ (i, 0) \suchthat i \in \interval{n} \right\},
\end{equation*}
where an extra row is added at the bottom of the skyline diagram of \(\alpha\).
The extra row is called the \vocab{basement} of \(\adg(\alpha)\).
\Cref{fig:skyline} provides an example of a skyline diagram and an augmented skyline diagram.

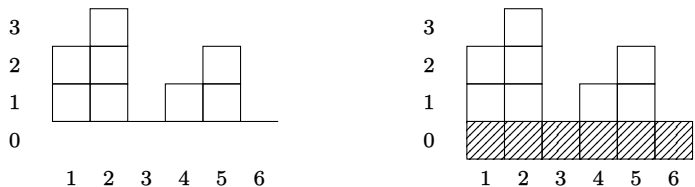
\begin{figure}[htbp]
	\centering
	\begin{tikzpicture}[x=.5cm, y=.5cm]
		\foreach \height [count = \column] in {2, 3, 0, 1, 2, 0}{
				\foreach \row in {0,...,\height}{
						\ifnum \row = 0
							\path (\column, \row) rectangle (\column + 1, \row + 1);
							\draw (\column, \row+1) -- (\column + 1, \row + 1);
						\else
							\draw[black] (\column, \row) rectangle (\column + 1, \row + 1);
						\fi
					}
				\node at (\column + .5, -.5) {\(\scriptstyle \column\)};
			}
		\foreach \row in {0, 1, 2, 3}{
				\node at (0, \row + .5) {\(\scriptstyle \row\)};
			}
	\end{tikzpicture}
	\qquad\qquad
	\begin{tikzpicture}[x=.5cm, y=.5cm]
		\foreach \height [count = \column] in {2, 3, 0, 1, 2, 0}{
				\foreach \row in {0,...,\height}{
						\ifnum \row = 0
							\draw[black,pattern=north east lines] (\column, \row) rectangle (\column + 1, \row + 1);
						\else
							\draw[black] (\column, \row) rectangle (\column + 1, \row + 1);
						\fi
					}
				\node at (\column + .5, -.5) {\(\scriptstyle \column\)};
			}
		\foreach \row in {0, 1, 2, 3}{
				\node at (0, \row + .5) {\(\scriptstyle \row\)};
			}
	\end{tikzpicture}
	\caption{
		The skyline diagram (left) and
		the augmented skyline diagram (right)
		of \(\alpha = \composition{2, 3, 0, 1, 2, 0}\).
		The basement of the augmented skyline diagram is shaded.
	}
	\label{fig:skyline}
\end{figure}

Let \(u = (i, r) \in \dg(\alpha)\).
The \vocab{left-arm set},
\vocab{right-arm set},
and \vocab{arm set} of \(u\) are defined as
\begin{align*}
	\leftArm(u) & = \left\{
		(i', r-1) \in \adg(\alpha)
		\suchthat
		i' < i \text{ and } \alpha_{i'} < \alpha_{i}
	\right\},\\
	\rightArm(u) & = \left\{
		(i', r) \in \dg(\alpha)
		\suchthat
		i' > i \text{ and } \alpha_{i'} \leq \alpha_{i}
	\right\}, \\
	\Arm(u) & = \leftArm(u) \cup \rightArm(u).
\end{align*}
The \vocab{leg statistic} and \vocab{arm statistic} of \(u\) are defined as \(\leg(u) = \alpha_i-r\)
and \(\arm(u) = \left|\Arm(u)\right|\).
The box \vocab{south} of \(u = (i, r)\) is the box \(\south{u} = (i, j-1) \in \adg(\alpha)\), that is, the box directly below \(u\).
See \cref{fig:leg-arm} for an example.

\begin{figure}[htbp]
	\begin{tikzpicture}[x=.5cm, y=.5cm]
		\def\shape{3, 2, 2, 4, 4, 0, 3, 3, 3, 4, 2, 1, 3}
		\xdef\reading{0}
		\foreach \height [count = \column] in \shape{
			\foreach \row in {0,...,\height}{
					\pgfmathparse{int(\reading + 1)}
					\xdef\reading{\pgfmathresult}
					\coordinate (B\reading) at (\column, \row);
				}
		}
		\def\filling{
			{ },{ },{ },{ },
			{ },{z},{ },
			{ },{z},{ },
			{ },{ },{ },{ },{ },
			{ },{ },{ },{ },{ },
			{ },
			{ },{ },{ },{ },
			{ },{d},{u},{x},
			{ },{ },{y},{ },
			{ },{ },{ },{ },{ },
			{ },{ },{y},
			{ },{ },
			{ },{ },{y},{ }}

		\foreach \i [count = \j] in \filling {
			\if\i u
				\draw[fill=black] (B\j) rectangle ++(1, 1);
			\else\if\i z
					\draw[pattern color=magenta, pattern=north east lines] (B\j) rectangle ++(1, 1);
				\else\if\i x
						\draw[pattern color=green!70!black, pattern=vertical lines] (B\j) rectangle ++(1, 1);
					\else\if\i y
							\draw[pattern color=blue, pattern=north west lines] (B\j) rectangle ++(1, 1);
						\else\if\i d
								\draw[pattern color=yellow!50!black, pattern=horizontal lines] (B\j) rectangle ++(1, 1);
							\else
							\draw[fill=white] (B\j) rectangle ++(1, 1);
						\fi\fi\fi\fi\fi
		}

	\end{tikzpicture}
	\caption{
		We show set of boxes counted by \(\leg(u)\) (\,\raisebox{-.7ex}{\begin{tikzpicture}[x=.45cm, y=.45cm] \protect\draw[pattern color=green!70!black, pattern=vertical lines] (0, 0) rectangle ++(1, 1); \end{tikzpicture}}\,),
		the left-arm set (\,\raisebox{-.7ex}{\begin{tikzpicture}[x=.45cm, y=.45cm] \protect\draw[pattern color=magenta, pattern=north east lines] (0, 1) rectangle ++(1, 1); \end{tikzpicture}}\,),
		the right-arm set (\,\raisebox{-.7ex}{\begin{tikzpicture}[x=.45cm, y=.45cm] \protect\draw[pattern color=blue, pattern=north west lines] (3, 4) rectangle ++(1, 1); \end{tikzpicture}}\,),
		and the south (\,\raisebox{-.7ex}{\begin{tikzpicture}[x=.45cm, y=.45cm] \protect\draw[pattern color=yellow!50!black, pattern=horizontal lines] (1, 4) rectangle ++(1, 1); \end{tikzpicture}}\,)
		of a box \(u = (8, 2)\) (\,\raisebox{-.7ex}{\begin{tikzpicture}[x=.45cm, y=.45cm] \protect\draw[fill=black] (5, 2) rectangle ++(1, 1); \end{tikzpicture}}\,)
		in the augmented skyline diagram of \(\alpha = \composition{3, 2, 2, 4, 4, 0, 3, 3, 3, 4, 2, 1, 3}\).
		Then \(\leg(u)=1\) and \(\arm(u)=5\).
	}
	\label{fig:leg-arm}
\end{figure}

Two boxes in the augmented diagram are \vocab{attacking} if they are in the same row, or if they are in consecutive rows
and the boxes in the higher row is strictly to the right of the box in the lower row.
In other words, two boxes are attacking if they are in one of the configurations shown in \cref{fig:attacking}.

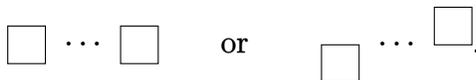
\begin{figure}[htbp]
	\begin{equation*}
		\vcenter{\hbox{
				\begin{tikzpicture}[x=.5cm, y=.5cm]
					\draw[black] (0, 0) rectangle (1, 1);
					\node at (2, .5) {\(\cdots\)};
					\draw[black] (3, 0) rectangle (4, 1);
				\end{tikzpicture}}}
		\qquad \text{or} \qquad
		\vcenter{\hbox{
				\begin{tikzpicture}[x=.5cm, y=.5cm]
					\draw[black, shift={(0,-.5)}] (0, 0) rectangle (1, 1);
					\node at (2, .5) {\(\cdots\)};
					\draw[black, shift={(0,.5)}] (3, 0) rectangle (4, 1);
				\end{tikzpicture}}}.
	\end{equation*}
	\caption{Configuration of attacking boxes in a same row (left) and in consecutive rows (right).}
	\label{fig:attacking}
\end{figure}

A \vocab{filling} of shape \(\alpha=\composition{\alpha_1,\ldots,\alpha_n}\) and basement \(\sigma \in \sym{n}\) is
a function \(T \colon \adg(\alpha) \to \interval{n}\) such that \(T(i, 0) = \sigma_i\) for all \(i \in \interval{n}\).
Let \(\Fill(\alpha, \sigma)\) denote the set of fillings of shape \(\alpha\) and basement \(\sigma\).
A filling \(T\in F(\alpha,\sigma)\)
is a \vocab{non-attacking}
if \(T(u_1) \neq T(u_2)\) for all pairs of attacking boxes \(u_1\) and \(u_2\) in \(\adg(\alpha)\).
Let \(\NAF(\alpha, \sigma)\) denote the set of non-attacking fillings of shape \(\alpha\) and basement \(\sigma\).
See \Cref{fig:non-attacking-filling} for an example.

\begin{figure}[htbp]
	\begin{tikzpicture}[x=.5cm, y=.5cm]
		\def\shape{2, 2, 0, 1}
		\def\filling{
			3, 1, 2,
			1, 2, 4,
			2,
			4, 4}
		\edef\reading{0}
		\foreach \height [count = \column] in \shape{
			\foreach \row in {0,...,\height}{
					\pgfmathparse{int(\reading + 1)}
					\xdef\reading{\pgfmathresult}
					\draw[black] (\column, \row) rectangle (\column + 1, \row + 1);
					\coordinate (B\reading) at (\column + .5, \row + .5);
				}
		}
		\foreach \i [count = \j] in \filling{ \node at (B\j) {\i}; }
	\end{tikzpicture}
	\caption{
		A non-attacking filling of
		shape \(\alpha = \composition{2, 2, 0, 1}\),
		basement \(\sigma = \window{3, 1, 2, 4}\),
		and \(\content(T) = \composition{1, 2, 0, 2}\).
	}
	\label{fig:non-attacking-filling}
\end{figure}

The \vocab{content} of a filling \(T \in \NAF(\alpha, \sigma)\)
is the composition
\(\content(T) = \composition{\beta_1, \beta_2, \ldots, \beta_n}\),
where \(\beta_i = |\{u\in\dg(\alpha) : T(u)=i\}|\)
is the number of boxes of \(T\) with label \(i\).
Note that boxes in the basement do not contribute to the content.
Let \(\NAF(\alpha, \sigma, \beta)\) denote
the set of non-attacking fillings
of shape \(\alpha\),
basement \(\sigma\),
and content \(\beta\).

Given \(a, b \in \interval{n}\),
let \(\ind{a, b}\) be the indicator function of \(a > b\),
that is,
\begin{equation*}
	\ind{a, b} =
	\begin{cases}
		1 & \text{if } a > b, \\
		0 & \text{otherwise}.
	\end{cases}
\end{equation*}
Given a non-attacking filling \(T \in \NAF(\alpha, \sigma)\),
we say that a box \(u \in \dg(\alpha)\)
is a \vocab{descent} in \(T\) if \(T(u) > T(\south{u})\).
Let \(\Des(T) \subseteq \dg(\alpha)\) denote the set of descents in \(T\).
The \vocab{major index} of \(T\) is
\begin{equation*}
	\maj(T)
	= \sum_{u \in \Des(T)} \leg(u) + 1.
\end{equation*}
See \cref{fig:major-index} for an example.
\begin{figure}[htbp]
	\begin{tikzpicture}[x=.5cm, y=.5cm]
		\def\shape{2, 2, 0, 1}
		\def\filling{
			3, 1, \textbf2,
			1, \textbf2, \textbf4,
			2,
			4, 4}
		\xdef\reading{0}
		\fill[black!20] (1,2) rectangle ++(1,1);
		\fill[black!20] (2,1) rectangle ++(1,1);
		\fill[black!20] (2,2) rectangle ++(1,1);
		\foreach \height [count = \column] in \shape{
			\foreach \row in {0,...,\height}{
					\pgfmathparse{int(\reading + 1)}
					\xdef\reading{\pgfmathresult}
					\draw[black] (\column, \row) rectangle (\column + 1, \row + 1);
					\coordinate (B\reading) at (\column + .5, \row + .5);
				}
		}
		\foreach \i [count = \j] in \filling{
			\node at (B\j) {\i};
		}
	\end{tikzpicture}
	\caption{
		The filling \(T\) above has three descents (shaded boxes).
		The leg statistics of the descents are: \(\leg(1,2)=\leg(2,2)=0\) and \(\leg(2,1)=1\).
		Thus, the major index of \(T\) is \(\maj(T) = 1 + 1 + 2 = 4\).
	}
	\label{fig:major-index}
\end{figure}

A \vocab{triple} is an ordered tuple \((u, v, w)\) of boxes in \(\adg(\alpha)\)
such that \(w = \south{u}\) and \(v \in \Arm(u)\).
If \(v \in \rightArm(u)\),
then the triple is of \vocab{type I},
and if \(v \in \leftArm(u)\),
then the triple is of \vocab{type II}.
The configurations of the boxes \(u\), \(v\), and \(w\)
when the tuple \((u, v, w)\) is of type I or type 2 are shown in \cref{fig:triples}.
Note that if \(\alpha\) is a partition, there are no triples of type II.
\begin{figure}[ht]
	\centering
	\begin{subfigure}[t]{0.45\textwidth}
		\centering
		\begin{ytableau}
			u & \none[\cdots] & v \\
			w
		\end{ytableau}
		\caption*{(\textsc{Type I})
		The column containing \(u\) and \(w\) is at least as high as the column containing \(v\).}
	\end{subfigure}\hspace{0.05\textwidth}
	\begin{subfigure}[t]{0.45\textwidth}
		\centering
		\begin{ytableau}
			\none & \none & u \\
			v & \none[\cdots] & w
		\end{ytableau}
		\caption*{(\textsc{Type II})
		The column containing \(u\) and \(w\) is strictly higher than the column containing \(v\).}
	\end{subfigure}
	\caption{Configurations for triples of type I (left) and type II (right).}
	\label{fig:triples}
\end{figure}

We define the \vocab{inversion} and \vocab{coinversion} statistics,
following \cite[Section~3]{HHL08}.
Given \(a, b, c \in \interval{n}\), such that \(a \neq b\) and \(b \neq c\), define
\begin{equation*}
	\ind{a, b, c} = \ind{a, b} + \ind{b, c} - \ind{a, c} \in \{0, 1\}.
\end{equation*}
Note that \(\ind{a,b,c}\in\{0,1\}\)
since it is not possible for \(\ind{a,b}=\ind{b,c}=1-\ind{a,c}\).

Equivalently,
\(\ind{a, b, c}\) is the indicator function of the statement
that \(a, b, c\) is cyclically weakly decreasing, that is,
\(\ind{a, b, c} = 1\) if and only if
\begin{equation*}
	a > b > c \quad \text{or} \quad b > c \geq a \quad \text{or} \quad c \geq a > b.
\end{equation*}
\Cref{fig:indicator} provides examples of the indicator function \(\ind{a, b, c}\).

\begin{figure}[htbp]
	\begin{tikzpicture}
		\draw[black] (0, 0) circle (1);
		\node[circle, fill=white] at (90:1) {\(3\)};
		\node[circle, fill=white] at (-30:1) {\(2\)};
		\node[circle, fill=white] at (-150:1) {\(1\)};
		\draw[->, green!60!black] (90:.6) arc (90:-150:.6);
		\node[green!60!black] at (0, -1.5) {\(\ind{3, 2, 1} = 1\)};
	\end{tikzpicture}
	\begin{tikzpicture}
		\draw[black] (0, 0) circle (1);
		\node[circle, fill=white] at (90:1) {\(1\)};
		\node[circle, fill=white] at (-30:1) {\(3\)};
		\node[circle, fill=white] at (-150:1) {\(2\)};
		\draw[->, green!60!black, rotate=-120] (90:.6) arc (90:-150:.6);
		\node[green!60!black] at (0, -1.5) {\(\ind{1, 3, 2} = 1\)};
	\end{tikzpicture}
	\begin{tikzpicture}
		\draw (0, 0) circle (1);
		\node[circle, fill=white] at (90:1) {\(2\)};
		\node[circle, fill=white] at (-30:1) {\(3\)};
		\node[circle, fill=white] at (-150:1) {\(2\)};
		\draw[->, green!60!black, rotate=-120] (90:.6) arc (90:-150:.6);
		\node[green!60!black] at (0, -1.5) {\(\ind{2, 3, 2} = 1\)};
	\end{tikzpicture}
	\begin{tikzpicture}
		\draw[black] (0, 0) circle (1);
		\node[circle, fill=white] at (90:1) {\(2\)};
		\node[circle, fill=white] at (-30:1) {\(1\)};
		\node[circle, fill=white] at (-150:1) {\(2\)};
		\draw[->, green!60!black, rotate=120] (90:.6) arc (90:-150:.6);
		\node[green!60!black] at (0, -1.5) {\(\ind{2, 1, 2} = 1\)};
	\end{tikzpicture}
	\begin{tikzpicture}
		\draw[black] (0, 0) circle (1);
		\node[circle, fill=white] at (90:1) {\(1\)};
		\node[circle, fill=white] at (-30:1) {\(2\)};
		\node[circle, fill=white] at (-150:1) {\(3\)};
		\draw[->, red!60!black] (105:.3) arc (105:-180:.3);
		\draw[red!60!black] (.3,.3) -- (-.3,-.3);
		\draw[red!60!black] (.3,-.3) -- (-.3,.3);
		\node[red!60!black] at (0, -1.5) {\(\ind{1, 2, 3} = 0\)};
	\end{tikzpicture}
	\caption{Examples of the indicator function \(\ind{a, b, c}\).}
	\label{fig:indicator}
\end{figure}
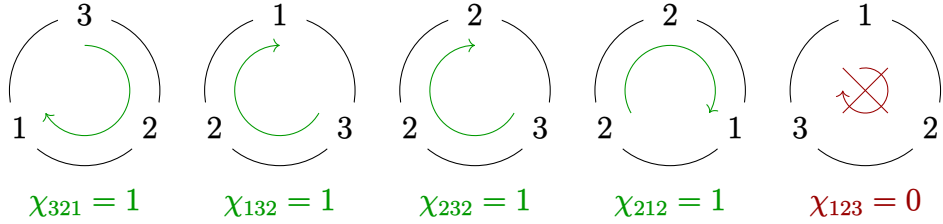

Given a non-attacking filling \(T \in \NAF(\alpha, \sigma)\),
a triple \((u, v, w)\) in \(\adg(\alpha)\) is
an \vocab{inversion triple} if \(\ind{T(u), T(v), T(w)} = 1\),
and is a \vocab{coinversion triple} otherwise.
Note that \(T(v) \neq T(u)\) and \(T(v) \neq T(w)\)
since \(v\) attacks both \(u\) and \(w\).
Let \(\inv(T)\) and \(\coinv(T)\) denote
the number of inversion triples and coinversion triples in \(T\),
respectively.
It is straightforward to check that these definitions of inversion triples
are equivalent to those in \cite{Ale19} and \cite{CMW22}.
See \cref{fig:inversion-coinversion} for an example.

\begin{figure}[htbp]
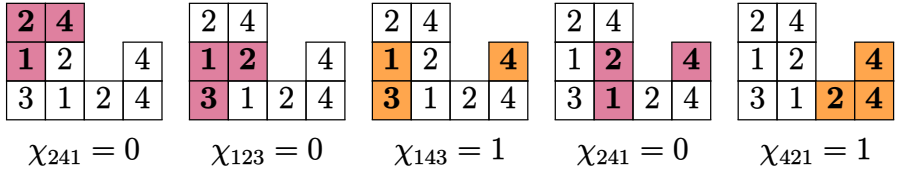

	\ytableausetup{boxsize=0.5cm}
	\colorlet{inversion}{orange!70!white}
	\colorlet{coinversion}{purple!50!white}
	\renewcommand{\arraystretch}{1.5}
	\begin{tabular}{ccccc}
		\begin{ytableau}
			*(coinversion) \textbf2 & *(coinversion) \textbf4 \\
			*(coinversion) \textbf1 & 2 & \none & 4 \\
			3 & 1 & 2 & 4
		\end{ytableau}
		&
		\begin{ytableau}
			2 & 4 \\
			*(coinversion) \textbf1 & *(coinversion) \textbf2 & \none & 4 \\
			*(coinversion) \textbf3 & 1 & 2 & 4
		\end{ytableau}
		&
		\begin{ytableau}
			2 & 4 \\
			*(inversion) \textbf1 & 2 & \none & *(inversion) \textbf4 \\
			*(inversion) \textbf3 & 1 & 2 & 4
		\end{ytableau}
		&
		\begin{ytableau}
			2 & 4 \\
			1 & *(coinversion) \textbf2 & \none & *(coinversion) \textbf4 \\
			3 & *(coinversion) \textbf1 & 2 & 4
		\end{ytableau}
		&
		\begin{ytableau}
			2 & 4 \\
			1 & 2 & \none & *(inversion) \textbf4 \\
			3 & 1 & *(inversion) \textbf2 & *(inversion) \textbf4
		\end{ytableau}
		\\
		\(\ind{2, 4, 1} = 0\)
		&
		\(\ind{1, 2, 3} = 0\)
		&
		\(\ind{1, 4, 3} = 1\)
		&
		\(\ind{2, 4, 1} = 0\)
		&
		\(\ind{4, 2, 1} = 1\)
	\end{tabular}
	\caption{
		We show all five triples of
		the non-attacking filling \(T\) in \cref{fig:non-attacking-filling};
		the first four are of type I, and the last one is of type II.
		There are two inversion triples (orange)
		and three coinversion triples (purple),
		which gives \(\inv(T) = 2\) and \(\coinv(T) = 3\).
	}
	\label{fig:inversion-coinversion}
\end{figure}

\subsection{Tableau formula for permuted-basement Macdonald polynomials}
\label{sec:tableau-formula}

For a non-attacking filling \(T \in \NAF(\alpha, \sigma)\),
define the \(\mathbf{x}\)-weight as
\begin{equation*}
	\mathbf{x}^T = \mathbf{x}^{\content(T)} = \prod_{u \in \dg(\alpha)} x_{T(u)},
\end{equation*}
the \(q, t\)-weight as
\begin{equation*}
	\wtqt(T)
	=
	q^{\maj(T)} t^{\coinv(T)}
	\prod_{\substack{u \in \dg(\alpha) \\ T(u) \neq T(\south{u})}}
	\frac{1 - t}{1 - q^{1 + \leg(u)} t^{1 + \arm(u)}},
\end{equation*}
and the weight as
\begin{equation*}
	\wt(T) = \mathbf{x}^T \wtqt(T).
\end{equation*}

The \(q,t\)-weight of a filling is the product of the weights
contributed by the descents
within each row and the inversions between every pair of rows.
Let \(L = \max_i \alpha_i\) be the total number of rows.
For \(0\leq r\leq L-1\),
define \(\dg[r]{(\alpha)}\) as
the set of boxes in the \(r\)\textsuperscript{th} row of \(\dg(\alpha)\).
Then, define
\begin{equation*}
	\maj[r + 1](T) = \sum_{u \in \Des[r + 1](T)} \leg(u) + 1,
\end{equation*}
where \(\Des[r+1](T)=\Des(T)\cap \dg[r+1]{(\alpha)}\).
Also, define \(\coinv[r, r+1](T)\) as
the number of inversion triples in \(T\)
within the pair of rows \(r\) and \(r + 1\).
Then, the \(r\)\textsuperscript{th} component
of the \(q, t\)-weight
of a non-attacking filling \(T \in \NAF(\alpha, \sigma)\)
is defined as
\begin{equation}
	\label{eq:rthcomponent}
	\wtqt[r](T) =
	q^{\maj[r + 1](T)}
	t^{\coinv[r, r + 1](T)}
	\prod_{\substack{u \in \dg[r + 1]{(\alpha)} \\ T(u) \neq T(\south{u})}}
	\frac{1 - t}{1 - q^{1 + \leg(u)} t^{1 + \arm(u)}},
\end{equation}
yielding the decomposition
\begin{equation*}
	\wtqt(T) = \prod_{r = 0}^{L-1} \wtqt[r](T).
\end{equation*}

\begin{example}
	Consider the non-attacking filling \(T\) in \cref{fig:non-attacking-filling}.
	The \(\mathbf{x}\)-weight of \(T\) is
	\begin{equation*}
		\mathbf{x}^T = x_1x_2^2x_4^2.
	\end{equation*}
	The \(0\)\textsuperscript{th} and \(1\)\textsuperscript{st} components
	of the \(q, t\)-weight of \(T\) are
	\begin{equation*}
		\wtqt[0](T) = \frac{ q^2 t^2 (1-t)^2 }{ (1 - q^2t^3) (1 - q^2t^2)},
		\qquad \text{ and } \qquad
		\wtqt[1](T) = \frac{ q^2 t^1 (1-t)^2 }{ (1 - qt^2) (1 - qt)}.
	\end{equation*}
	Then the \(q, t\)-weight of \(T\) is \(\wtqt(T)=\wtqt[0](T)\wtqt[1](T)\),
	and the weight of \(T\) is
	\begin{equation*}
		\wt(T) = x_1x_2^2x_4^2 \frac{ q^4 t^3 (1-t)^4 }{ (1 - q^2t^3) (1 - q^2t^2) (1 - qt^2) (1 - qt)}.
	\end{equation*}
\end{example}

\begin{theorem}[{\cite[Definition~4.4.2]{Fer11}}]
	\label{thm:tableau-formula}
	For \(\alpha=\composition{\alpha_1,\ldots,\alpha_n}\) and \(\sigma\in\sym{n}\),
	the permuted-basement Macdonald polynomial \(E_\alpha^\sigma\) is given by
	\begin{equation*}
		E^{\sigma}_{\alpha}(\mathbf{x}; q, t) = \sum_{T \in \NAF(\alpha, \sigma)} \wt(T).
	\end{equation*}
\end{theorem}

\begin{example}
	Let \(n = 4\), \(\alpha = \composition{1, 1, 0, 1}\) and \(\sigma = \window{2, 4, 1, 3}\).
	The four non-attacking fillings in \(\NAF(\alpha, \sigma)\)
	and their weights are shown in \cref{fig:fillings-1101-2413}.
	Using \cref{thm:tableau-formula}, the permuted-basement Macdonald polynomial
	\(E_{\composition{1, 1, 0, 1}}^{\window{2, 4, 1, 3}}(\mathbf{x}; q, t)\)
	is the sum of the weights of these fillings,
	agreeing with the computation in \cref{example:permuted-basement-macdonald-polynomial}.
	\begin{figure}[htbp]
		\ytableausetup{boxsize=0.5cm}
		\renewcommand{\arraystretch}{2}
		\setlength{\tabcolsep}{1em}
		\begin{tabular}{cccc}
			\begin{ytableau}
				2 & 4 & \none & 3 \\
				2 & 4 & 1 & 3 \\
			\end{ytableau}
			&
			\begin{ytableau}
				1 & 4 & \none & 3 \\
				2 & 4 & 1 & 3 \\
			\end{ytableau}
			&
			\begin{ytableau}
				4 & 1 & \none & 3 \\
				2 & 4 & 1 & 3 \\
			\end{ytableau}
			&
			\begin{ytableau}
				2 & 1 & \none & 3 \\
				2 & 4 & 1 & 3 \\
			\end{ytableau}
			\\
			\(x_2x_3x_4\)
			&
			\(x_1x_3x_4 \frac{1 - t}{1 - qt^3}\)
			&
			\(x_1x_3x_4 \frac{qt^2 (1 - t)^2}{(1 - q t^3)(1 - q t^2)}\)
			&
			\(x_1x_2x_3 \frac{t (1-t)}{1-qt^2}\)
		\end{tabular}
		\caption{All fillings in \(\NAF(\composition{1, 1, 0, 1}, \window{2, 4, 1, 3})\) and their weights.}
		\label{fig:fillings-1101-2413}
	\end{figure}
\end{example}

\subsection{Probabilistic bijections}

Probabilistic bijections,
also referred to as \emph{bijectivization} or \emph{coupling} by \cite{BP19},
generalize weight-preserving bijections.
We follow the exposition given in~\cite{AF22}.

\begin{definition}[Probability map]
	Let \(\algebra\) be an algebra.
	Let \(\mathbf{T}\) and \(\mathbf{U}\) be sets.
	A \vocab{probability map} from \(\mathbf{T}\) to \(\mathbf{U}\) is
	a function \(\prob{} \colon \mathbf{T} \times \mathbf{U} \to \algebra\) such that
	\begin{equation*}
		\sum_{U \in\mathbf{U}} \prob{}(T, U) = 1
	\end{equation*}
	for every \(T \in \mathbf{T}\).
\end{definition}

Note that these do not formally correspond to probability maps
in the sense of probability theory
because the values of \(\prob{}\)
are not required to be in the interval \([0, 1]\) or even real numbers.

\begin{definition}[Probabilistic bijection {\cite{AF22}}]\label{def:prob}
	Let \(\algebra\) be an algebra.
	Let \(\mathbf{T}\) and \(\mathbf{U}\) be finite sets
	equipped with weight functions
	\(\wt_{\mathbf{T}} \colon \mathbf{T} \to \algebra\) and
	\(\wt_{\mathbf{U}} \colon \mathbf{U} \to \algebra\).
	A \vocab{probabilistic bijection} between
	\( (\mathbf{T}, \wt_{\mathbf{T}}) \) and
	\( (\mathbf{U}, \wt_{\mathbf{U}}) \) is
	a pair of probability maps
	\(\prob{\mathbf{T}} \colon \mathbf{T} \times \mathbf{U} \to \algebra\) and
	\(\prob{\mathbf{U}} \colon \mathbf{U} \times \mathbf{T} \to \algebra\)
	such that
	\begin{equation*}
		\wt_{\mathbf{T}}(T) \prob{\mathbf{T}}(T, U)
		= \wt_{\mathbf{U}}(U) \prob{\mathbf{U}}(U, T)
	\end{equation*}
	for every \(T \in \mathbf{T}\) and \(U \in \mathbf{U}\).
\end{definition}

A weight-preserving bijection \( f \colon \mathbf{T} \to \mathbf{U} \)
canonically induces a probabilistic bijection
between \((\mathbf{T}, \wt_{\mathbf{T}})\) and \((\mathbf{U}, \wt_{\mathbf{U}})\)
by setting
\begin{equation*}
	\prob{\mathbf{T}}(T, U) = \prob{\mathbf{U}}(U, T) =
	\begin{cases}
		1 & \text{if } f(T) = U, \\
		0 & \text{otherwise},
	\end{cases}
\end{equation*}
explaining the choice of terminology.
Moreover, as shown in \cref{proposition:equal-weight-sum},
this generalization retains a key property of weight-preserving bijections:
the sum of the weights of \(\mathbf{T}\) and of \(\mathbf{U}\) are equal.

\begin{proposition}
	\label{proposition:equal-weight-sum}
	Let \(\prob{\mathbf{T}} \colon \mathbf{T} \times \mathbf{U} \to \algebra\)
	and \(\prob{\mathbf{U}} \colon \mathbf{U} \times \mathbf{T} \to \algebra\)
	form a probabilistic bijection between
	\( (\mathbf{T}, \wt_{\mathbf{T}}) \) and
	\( (\mathbf{U}, \wt_{\mathbf{U}}) \).
	Then,
	\begin{equation*}
		\sum_{T \in \mathbf{T}} \wt(T) = \sum_{U \in\mathbf{U}} \wt(U).
	\end{equation*}
\end{proposition}

\begin{proof}
	We compute
	\begin{align*}
		\sum_{T \in \mathbf{T}} \wt_{\mathbf{T}}(T)
		&= \sum_{T \in \mathbf{T}} \sum_{U \in \mathbf{U}}
			\wt_{\mathbf{T}}(T) \prob{\mathbf{T}}(T, U) \\
		&= \sum_{U \in \mathbf{U}} \sum_{T \in \mathbf{T}}
			\wt_{\mathbf{U}}(U) \prob{\mathbf{U}}(U, T)
		= \sum_{U \in \mathbf{U}} \wt_{\mathbf{U}}(U). \qedhere
	\end{align*}
\end{proof}

A Markov-theoretical interpretation of a probabilistic bijection is that
it defines a bipartite Markov chain on the state space \( \mathbf{T} \cup \mathbf{U} \)
with transition probabilities given by the function \( \prob{} \) under a specific balance condition.
In the case of \cref{def:prob}, this condition is detailed balance,
but a probabilistic bijection can also be defined using a more general balance condition.

\section{A probabilistic bijection for non-attacking fillings}
\label{sec:probabilistic}

In this section,
we define a map on non-attacking fillings
that allows us to construct a probabilistic bijection between
\(\NAF(\alpha,\sigma)\) and \(\NAF(\alpha,\sigma s_i)\)
in order to prove \cref{theorem:main} in \cref{sec:proof}.

In \cite{CHMMW22},
a weight-preserving bijection was established between
fillings in \(\Fill(\lambda)\) with statistics \(\maj\) and \(\inv\),
corresponding to swapping two adjacent entries in the bottom row within columns of the same height.
This was done using an operator \(\tau_i\),
originally introduced in \cite{MR04},
on fillings in \(\Fill(\lambda)\).
The bijection established the equality of generating functions
over weighted fillings for the sets
\(\{T\in\Fill(\lambda) \suchthat \text{bottom row of \(T\) is \(w\)}\}\)
and
\(\{T\in\Fill(\lambda) \suchthat \text{bottom row of \(T\) is \(w s_i\)}\}\)
for an index \(i\) such that \(\lambda_i = \lambda_{i+1}\).
The result was used to obtain a compact formula for the modified Macdonald polynomials.

In our case,
proving that \(E_\alpha^\sigma=E_\alpha^{\sigma s_i}\) combinatorially
requires showing that the weight generating functions
over \(\NAF(\alpha,\sigma)\) and \(\NAF(\alpha,\sigma s_i)\) are equal,
which also amounts to finding a bijection on tableaux
that swaps basement entries in columns \(i\) and \(i+1\).
However,
the operator \(\tau_i\) used in \cite{CHMMW22}
does not preserve the non-attacking property of tableaux,
so the same strategy could not be applied.
This issue was resolved in \cite{Man24}
by introducing a probabilistic operator \(\widehat{\tau}_i\)
for non-attacking tableaux with partition shape,
in order to obtain a compact formula for \(P_{\lambda}\).
The latter operator is the map we use to define our probabilistic bijection.

There are two key differences between our setting and the one in \cite{Man24}.
First, our map applies to tableaux with composition shapes,
whereas \cite{Man24} considers only partition shapes.
This distinction allows \cite{Man24} to restrict attention to type I triples,
while we consider both types.
Second, our fillings include basements,
eliminating the need for the so-called \emph{degenerate} triples in \cite{Man24}.
Hence, the probability map in \cref{definition:prob}
and the surrounding lemmas are adapted from the original to fit our setting.

Throughout this section,
let \(\alpha=\composition{\alpha_1,\ldots,\alpha_n}\) be a composition,
let \(\sigma \in \sym{n}\) be a permutation,
and let \(i \in \interval{n-1}\) such that \(\alpha_i = \alpha_{i+1}\).

\begin{definition}[{\(\swap[r]{i}\) and \(\sswap[0, r]{i}\)}]
	\label{definition:swap}
	Given a filling \(T\) of shape \(\alpha\) and \(r \in \interval{0, \alpha_i}\),
	let \(\swap[r]{i}(T)\) denote the filling obtained by
	swapping the entries in boxes \((i, r)\) and \((i+1, r)\) in \(T\).
	Moreover,
	let \(\sswap[0, r]{i} = \swap[r]{i} \circ \cdots \circ \swap[1]{i} \circ \swap[0]{i}\),
	which represents a sequence of swaps
	in the first \(r+1\) rows within the columns \(i\) and \(i+1\).
\end{definition}

Note that if \(T \in \Fill(\alpha, \sigma)\),
then \(\sswap[0, r]{i}(T) \in \Fill(\alpha, \sigma s_i)\) for any \(r\geq 0\).
Also, the fillings \(T\) and \(\sswap[0, r]{i}(T)\) have the same content.
However, \(\swap[r]{i}\) and \(\sswap[0, r]{i}\)
do not in general preserve the property of a filling being non-attacking.

\newcommand{\pp}[2][]{\rho_{#2}^{(#1)}}
\begin{definition}[{\(\pp[r]{i}\)}]
	\label{definition:propagate}
	Let \(T \in \NAF(\alpha, \sigma)\) be a non-attacking filling.
	Let \(r \in \interval{0, \alpha_i-1}\).
	We define \(\pp[r]{i}(T) \in \rationals(q, t)\).
	Let \(a = T(i, r)\), \(b = T(i+1, r)\), \(c = T(i, r+1)\), and \(d = T(i+1, r+1)\).
	Let \(A = \arm(i+1, r+1)\) and \(\ell = \leg(i+1, r+1)\).

	If \(|\{a, b, c, d\}| = 4\), then we split into two cases.
	\begin{enumerate}
		\item \label{item:pp-cda-cdb}
			If \(\ind{c, d, a} = \ind{c, d, b}\), then \(\pp[r]{i}(T) = 0\).
		\item \label{item:pp-cda-dcb}
			If \(\ind{c, d, a} = \ind{d, c, b}\), then \(\pp[r]{i}(T) = 1\).
	\end{enumerate}

	If \(|\{a, b, c, d\}| = 3\),
	the non-attacking condition implies that \(a \neq b\),
	\(a \neq d\) and \(c \neq d\),
	so we have three cases.
	\begin{enumerate}[resume]
		\item \label{item:pp-bc}
			If \(b = c\), then \(\pp[r]{i}(T) = 0\).
		\item \label{item:pp-bd}
			If \(b = d\), then \(\pp[r]{i}(T) = 1\).
		\item \label{item:pp-ac}
			If \(a = c\), then
			\begin{equation*}
				\pp[r]{i}(T) =
				t^{1 - \ind{d, a, b}}
				\frac{1 - q^{\ell+1}t^{A+1}}{1-q^{\ell+1}t^{A+2}}.
			\end{equation*}
	\end{enumerate}

	If \(|\{a, b, c, d\}| = 2\),
	then we have one case.
	\begin{enumerate}[resume]
		\item \label{item:pp-ac-bd}
		If \(a = c\) and \(b = d\), then \(\pp[r]{i}(T) = 1\).
	\end{enumerate}
\end{definition}

\Cref{table:propagate} summarizes the values of \(\pp[r]{i}(T)\)
for \(T \in \NAF(\alpha, \sigma)\) and \(r \in \interval{0, \alpha_i-1}\).
\begin{table}[htbp]
	\centering
	\ytableausetup{boxsize=1em}
	\setcellgapes{0pt}
	\makegapedcells
	\begin{tabular}[c]{rccc}
		\toprule
		& \(T\) & \(\pp[r]{i}(T)\) & \(1 - \pp[r]{i}(T)\) \\
		\midrule
		\ref{item:pp-cda-cdb}
		&
		\begin{tabular}[t]{c}
		\begin{ytableau}
			\none[\scriptstyle r+1] & \none & c & d \\
			\none[\scriptstyle r]   & \none & a & b \\
		\end{ytableau} \\
		{\scriptsize \(\ind{c, d, a} = \ind{c, d, b}\)}
		\end{tabular}
		& \(0\) & \(1\) \\ \midrule
		\ref{item:pp-cda-dcb}
		&
		\begin{tabular}[t]{c}
		\begin{ytableau}
			\none[\scriptstyle r+1] & \none & c & d \\
			\none[\scriptstyle r]   & \none & a & b \\
		\end{ytableau}\\
		{\scriptsize \(\ind{c, d, a} = \ind{d, c, b}\)}
		\end{tabular}
		& \(1\) & \(0\) \\ \midrule
		\ref{item:pp-bc}
		& \begin{ytableau}
				\none[\scriptstyle r+1] & \none & b & d \\
				\none[\scriptstyle r]   & \none & a & b \\
			\end{ytableau}
		& \(0\) & \(1\) \\ \midrule
		\ref{item:pp-bd}
		& \begin{ytableau}
				\none[\scriptstyle r+1] & \none & c & b \\
				\none[\scriptstyle r]   & \none & a & b \\
			\end{ytableau}
		& \(1\) & \(0\) \\ \midrule
		\ref{item:pp-ac}
		& \begin{ytableau}
				\none[\scriptstyle r+1] & \none & a & d \\
				\none[\scriptstyle r]   & \none & a & b \\
			\end{ytableau}
		& \(t^{1 - \ind{d, a, b}} \frac{1 - q^{\ell+1}t^{A+1}}{1-q^{\ell+1}t^{A+2}}\)
		& \((q^{\ell+1}t^{A+1})^{\ind{d, a, b}} \frac{1-t}{1 - q^{\ell+1}t^{A+2}}\) \\ \midrule
		\ref{item:pp-ac-bd}
		& \begin{ytableau}
				\none[\scriptstyle r+1] & \none & a & b \\
				\none[\scriptstyle r]   & \none & a & b \\
				\none & \none[\scriptstyle i] & \none[\scriptstyle i+1]
			\end{ytableau}
		& \(1\) & \(0\) \\
		\bottomrule
	\end{tabular}
	\caption{
		The values of \(\pp[r]{i}(T)\) and \(1 - \pp[r]{i}(T)\)
		for \(T \in \NAF(\alpha, \sigma)\)
		and \(r \in \interval{0, \alpha_i-1}\),
		on each case of \cref{definition:propagate}.
	}
	\label{table:propagate}
\end{table}

For ease of computation,
we note that, in case \ref{item:pp-ac},
we have
\begin{equation*}
	1 - \pp[r]{i}(T)
	=
	(q^{\ell+1}t^{A+1})^{\ind{d, a, b}}
	\frac{1-t}{1 - q^{\ell+1}t^{A+2}},
\end{equation*}
which is straightforward to verify by
splitting into cases of \(\ind{d, a, b}\) being \(0\) or \(1\)
(compare to {\cite[Lemma 3.3]{Man24}}).

As a convention, we set \(\pp[\alpha_i]{i}(T) = 0\).

\begin{definition}[\(\prob{i}\)] \label{definition:prob}
	The map
	\begin{equation*}
		\prob{i} \colon
		\NAF(\alpha, \sigma) \times \Fill(\alpha, \sigma s_i)
		\to \rationals(q, t)
	\end{equation*}
	is defined for \(T \in \NAF(\alpha, \sigma)\) and
	\(U \in \Fill(\alpha, \sigma s_i)\) by letting
	\begin{equation}\label{eq:prob factors}
		\prob{i}(T, U) =
		\left( \prod_{r = 0}^{h-1} \pp[r]{i}(T) \right)
		\left( 1 - \pp[h]{i}(T) \right)
	\end{equation}
	if there exists \(h \in \interval{0, \alpha_i}\)
	such that \(U = \sswap[0, h]{i}(T)\),
	and \(\prob{i}(T, U) = 0\) otherwise.
\end{definition}

\begin{proposition}
	The map
	\(\prob{i} \colon \NAF(\alpha, \sigma) \times \Fill(\alpha, \sigma s_i) \to \rationals(q, t)\)
	is a probability map, that is,
	for all \(T \in \NAF(\alpha, \sigma)\),
	\begin{equation*}
		\sum_{U \in \Fill(\alpha, \sigma s_i)} \prob{i}(T, U) = 1.
	\end{equation*}
\end{proposition}

\begin{proof}
	Using that \(\pp[\alpha_i]{i}(T) = 0\),
	it is straightforward to compute
	\begin{equation*}
		\sum_{h=0}^{\alpha_i}
		\left( \prod_{r = 0}^{h-1} \pp[r]{i}(T) \right)
		\left( 1 - \pp[h]{i}(T) \right)
		= 1,
	\end{equation*}
	as required.
\end{proof}

\begin{example}
	Let \(T \in \NAF(\composition{3, 4, 4, 0, 0}, \window{5, 1, 3, 4, 2})\)
	be the non-attacking filling
	\begin{equation*}
		\ytableausetup{aligntableaux=center}
		\begin{ytableau}
			\none & 3 & 2 \\
			4 & 1 & 2 \\
			3 & 4 & 1 \\
			3 & 4 & 2 \\
			5 & 1 & 3 & 4 & 2
		\end{ytableau}.
	\end{equation*}
	The values of \(\pp[r]{2}(T)\) and \(1 - \pp[r]{2}(T)\) are
	\begin{equation*}
		\begin{aligned}
			\pp[0]{2}(T) &= 1,
			& 1-\pp[0]{2}(T) &= 0, \\
			\pp[1]{2}(T) &= \frac{1 - q^{3}t^{1}}{1 - q^{3}t^{2}},
			& 1-\pp[1]{2}(T) &= \frac{q^{3}t^{1} - q^{3}t^{2}}{1 - q^{3}t^{2}}, \\
			\pp[2]{2}(T) &= 0,
			& 1-\pp[2]{2}(T) &= 1, \\
			\pp[3]{2}(T) &= 1,
			& 1-\pp[3]{2}(T) &= 0, \\
			\pp[4]{2}(T) &= 0,
			& 1-\pp[4]{2}(T) &= 1.
		\end{aligned}
	\end{equation*}
	Hence, the values of \(\prob{2}{T, \sswap[0, h]{2}(T)}\) for \(h \in \interval{0, 4}\) are
	\begin{equation*}
		\begin{aligned}
			\prob{2}\left(T, {\sswap[0, 0]{2}(T)}\right) &= 0, \\
			\prob{2}\left(T, {\sswap[0, 1]{2}(T)}\right) &= \frac{q^{3}t^{1} - q^{3}t^{2}}{1 - q^{3}t^{2}}, \\
			\prob{2}\left(T, {\sswap[0, 2]{2}(T)}\right) &= \frac{1 - q^{3}t^{1}}{1 - q^{3}t^{2}}, \\
			\prob{2}\left(T, {\sswap[0, 3]{2}(T)}\right) &= 0, \\
			\prob{2}\left(T, {\sswap[0, 4]{2}(T)}\right) &= 0.
		\end{aligned}
	\end{equation*}
	The two nonzero values of \(\prob{2}(T, U)\) are obtained for
	\(U = \sswap[0, 1]{2}(T)\) and
	\(U = \sswap[0, 2]{2}(T)\),
	which are, respectively,
	\begin{equation*}
		\ytableausetup{aligntableaux=center}
		\begin{ytableau}
			\none & 3 & 2 \\
			4 & 1 & 2 \\
			3 & 4 & 1 \\
			3 & 2 & 3 \\
			5 & 3 & 1 & 4 & 2
		\end{ytableau}
		\quad \text{and} \qquad
		\begin{ytableau}
			\none & 3 & 2 \\
			4 & 1 & 2 \\
			3 & 1 & 4 \\
			3 & 2 & 4 \\
			5 & 3 & 1 & 4 & 2
		\end{ytableau}.
	\end{equation*}
\end{example}

By construction, the \(\prob{i}\) is nonzero only for pairs of non-attacking fillings.
\begin{proposition} \label{proposition:restrict-to-naf-naf}
	Let \(T \in \NAF(\alpha, \sigma)\) be a non-attacking filling.
	Let \(U \in \Fill(\alpha, \sigma s_i)\) be a filling.
	If \(\prob{i}(T, U) \neq 0\), then \(U\) is non-attacking.
\end{proposition}

\begin{proof}
	Note that attacking pairs not involving columns \(i\) and \(i+1\)
	remain unchanged between \(T\) and \(U\),
	and attacking pairs involving one box in columns \(i\) or \(i+1\) in \(T\)
	correspond to attacking pairs with a box in the same columns in \(U\),
	though their positions may be swapped between the two columns.
	Hence,
	the only distinction in attacking pairs between \(T\) and \(U\)
	comes from the attacking boxes in both columns \(i\) and \(i+1\).

	If \(\prob{i}(T, U) \neq 0\),
	there exists \(h \in \interval{0, \alpha_i}\) such that \(U = \sswap[0, h]{i}(T)\),
	and thus by \eqref{eq:prob factors},
	\(\pp[r]{i}(T) \neq 0\) for all \(r \in \interval{0, h-1}\) and \(\pp[h]{i}(T) \neq 1\).
	Therefore,
	it suffices to show that if \(T\) is non-attacking,
	then \(U(i, r) \neq U(i+1, r+1)\) for all \(r \in \interval{0, h}\).
	For \(r \in \interval{0, h-1}\),
	we have \(\pp[r]{i}(T) \neq 0\), so the configuration of \cref{item:pp-bc} does not occur;
	thus, \(U(i, r) = T(i+1, r) \neq T(i, r+1) = U(i+1, r+1)\).
	For \(r = h\),
	we have \(\pp[h]{i}(T) \neq 1\),
	so configurations of \cref{item:pp-bd} or \cref{item:pp-ac-bd} does not occur;
	thus, \(U(i, h) = T(i+1, h) \neq T(i+1, h+1) = U(i+1, h+1)\).
	Therefore, \(U\) is indeed non-attacking.
\end{proof}

\begin{corollary}
	The restriction
	\begin{equation*}
		\prob{i} \colon \NAF(\alpha, \sigma) \times \NAF(\alpha, \sigma s_i) \to \rationals(q, t)
	\end{equation*}
	is a probability map.
\end{corollary}

As an abuse of notation, we write \(\prob{i}\) for both
the function on \(\NAF(\alpha, \sigma) \times \Fill(\alpha, \sigma s_i)\) and
the function on \(\NAF(\alpha, \sigma s_i) \times \NAF(\alpha, \sigma)\).
Moreover, we also write \(\prob{i}\) for both
the function on \(\NAF(\alpha, \sigma) \times \NAF(\alpha, \sigma s_i)\) and
the function on \(\NAF(\alpha, \sigma s_i) \times \NAF(\alpha, \sigma)\).

The main result of this section is \cref{theorem:wt-prob},
which generalizes \cite[Lemma 4.4]{Man24} to our setting.

\begin{theorem}[adapted from {\cite[Lemma 4.4]{Man24}}]
	\label{theorem:wt-prob}
	The pair of maps
	\begin{gather*}
		{\prob{i}} \colon \NAF(\alpha, \sigma) \times \NAF(\alpha, \sigma s_i) \to \rationals(q, t) \\
		{\prob{i}} \colon \NAF(\alpha, \sigma s_i) \times \NAF(\alpha, \sigma) \to \rationals(q, t)
	\end{gather*}
	defines a probabilistic bijection
	between \(\NAF(\alpha, \sigma)\) and \(\NAF(\alpha, \sigma s_i)\)
	with respect to the weight function \(\wtqt\).
	In other words,
	for all \(T \in \NAF(\alpha, \sigma)\)
	and \(U \in \NAF(\alpha, \sigma s_i)\),
	\begin{equation*}
		\wtqt(U)\prob{i}(U, T) = \wtqt(T)\prob{i}(T, U).
	\end{equation*}
\end{theorem}

For organization purposes,
we split the proof of \cref{theorem:wt-prob} into several lemmas.

\begin{lemma} \label{lemma:useful-ind-five-variables}
	Let \(a, b, c, d, f \in [n]\).
	Then,
	\begin{equation*}
		\ind{c, f, a} + \ind{d, f, b} - \ind{c, f, b} - \ind{d, f, a}
		= \ind{c, b} + \ind{d, a} - \ind{c, a} - \ind{d, b}.
	\end{equation*}
\end{lemma}
\cref{lemma:useful-ind-five-variables} follows from a direct computation.

\begin{lemma} \label{lemma:wtqt-prob-balanced-r}
	Let \(T \in \NAF(\alpha, \sigma)\) be a non-attacking filling.
	Let \(U = \sswap[0, h]{i}(T)\) for some \(h \in \interval{0, \alpha_i}\).
	Then, for all \(r \in \interval{0, h - 1}\),
	\begin{equation} \label{eq:wtqt-prob-balanced-r}
		\wtqt[r](U)\pp[r]{i}(U) = \wtqt[r](T)\pp[r]{i}(T).
	\end{equation}
\end{lemma}

\begin{proof}
	Let \(r \in \interval{0, h-1}\).
	Let
	\begin{align*}
		c & = T(i, r+1)=U(i+1, r+1), &
		d & = T(i+1, r+1)=U(i, r+1), \\
		a & = T(i, r)=U(i+1, r), \text{ and } &
		b & = T(i+1, r)=U(i, r).
	\end{align*}
	Since \(U\) is non-attacking, we have \(b \neq c\).
	We note that \(\maj[r + 1](T) = \maj[r + 1](U)\).
	Moreover,
	triples not involving columns \(i\) and \(i+1\)
	are preserved between \(T\) and \(U\),
	triples involving column \(i\) in \(T\)
	correspond to triples involving column \(i+1\) in \(U\),
	and triples involving column \(i+1\) in \(T\)
	correspond to triples involving column \(i\) in \(U\).
	Hence, the only distinction in the coinversion and inversion statistics between \(T\) and \(U\)
	comes from the triples involving both columns \(i\) and \(i+1\).
	Therefore,
	\begin{equation*}
		\coinv[r, r + 1](T) - \coinv[r, r + 1](U)
		=
		\inv[r, r + 1](U) - \inv[r, r + 1](T)
		=
		\ind{d, c, b} - \ind{c, d, a}.
	\end{equation*}
	We show that \eqref{eq:wtqt-prob-balanced-r} holds
	for all cases of \((a,b,c,d)\) with \(b\neq c\).

	First, suppose \(|\{a, b, c, d\}| = 4\).
	Then,
	the sets \(\{u \in \Arm(i+1, r+1) \suchthat T(u) \neq T(\south{u})\}\)
	and \(\{u \in \Arm(i+1, r+1) \suchthat U(u) \neq U(\south{u})\}\) are equal.
	\begin{itemize}[left=0pt]
		\item On the one hand, if \(\ind{c, d, a} = \ind{c, d, b}\),
			then \(\pp[r]{i}(T) = 0\) by \cref{item:pp-cda-cdb}.
			Then, it also holds that \(\ind{d, c, b} = \ind{d, c, a}\),
			and hence \(\pp[r]{i}(U) = 0\) by \cref{item:pp-cda-cdb}.
		\item On the other hand, if \(\ind{c, d, a}=1-\ind{c, d, b} = \ind{d, c, b}\),
			then \(\pp[r]{i}(T) = 1\) by \cref{item:pp-cda-dcb}.
			Then, it also holds that \(\ind{d, c, b} = \ind{c, d, a}\),
			and hence \(\pp[r]{i}(U) = 1\) by \cref{item:pp-cda-dcb}.
			For the (co)inversion number, we have
			\(\coinv[r, r + 1](T) - \coinv[r, r + 1](U) = \ind{d, c, b} - \ind{c, d, a} = 0\).
	\end{itemize}
	Thus, in both cases for \(\ind{c,d,a}\),
	\eqref{eq:wtqt-prob-balanced-r} holds for \(|\{a,b,c,d\}| = 4\).

	Next, suppose \(|\{a, b, c, d\}| = 3\).
	Since \(b\neq c\), we either have \(b=d\) or \(a=c\).
	Without loss of generality, assume \(b=d\),
	since swapping \(T\) and \(U\) gives the other case.
	By \cref{item:pp-bd}, we have
	\(\pp[r]{i}(T) = 0\).
	Moreover,
	\begin{equation*}
		\pp[r]{i}(U)
		=
		t^{1 - \ind{c, b, a}}
		\frac{(1 - q^{\ell+1}t^{A+1})}{1 - q^{\ell+1}t^{A+2}},
	\end{equation*}
	where \(A = \arm(i+1, r+1)\), by \cref{item:pp-ac}.
	For the (co)inversion number, we have
	\(\coinv[r, r + 1](T) - \coinv[r, r + 1](U) = \ind{b, c, b} - \ind{c, b, a} = 1 - \ind{c, b, a}\).
	Finally, the ratio
	\begin{equation*}
		\left(
			\prod_{\substack{u \in \dg[r + 1]{\alpha} \\ T(u) \neq T(\south{u})}}
			\frac{1 - t}{1 - q^{\leg(u)+1} t^{ \arm(u)+1}}
		\right)
		\big/
		\left(
			\prod_{\substack{u \in \dg[r + 1]{\alpha} \\ U(u) \neq U(\south{u})}}
			\frac{1 - t}{1 - q^{\leg(u)+1} t^{\arm(u)+1}}
		\right)
	\end{equation*}
	only differs in the boxes \((i,r+1)\) and \((i+1,r+1)\), and thus equals
	\begin{equation*}
		\left(
		\frac{1-t}{1 - q^{\ell+1}t^{A+2}}
		\right)
		\big/
		\left(
		\frac{1-t}{1 - q^{\ell+1}t^{A+1}}
		\right)
		= \frac{1 - q^{\ell+1}t^{A+1}}{1 - q^{\ell+1}t^{A+2}},
	\end{equation*}
	where \(A = \arm(i+1, r+1)\).
	Therefore, we compute
	\begin{equation*}
		\frac{\wtqt[r](T)}{\wtqt[r](U)} =
		t^{1 - \ind{c, b, a}} \frac{1 - q^{\ell+1}t^{A+1}}{1 - q^{\ell+1}t^{A+2}}
		= \pp[r]{i}(T, U),
	\end{equation*}
	confirming \eqref{eq:wtqt-prob-balanced-r}.

	Finally, if \(|\{a, b, c, d\}| = 2\),
	then \(a = c\) and \(b = d\).
	Hence, \(\pp[r]{i}(T) = \pp[r]{i}(U) = 1\), by \cref{item:pp-ac-bd}.
	For the (co)inversion number, we have
	\(\coinv[r, r + 1](T) - \coinv[r, r + 1](U) = \ind{b, a, b} - \ind{a, b, a} = 1 - 1 = 0\).
	Moreover, the sets \(\{u \in \Arm(i+1, r+1) \suchthat T(u) \neq T(\south{u})\}\)
	and \(\{u \in \Arm(i+1, r+1) \suchthat U(u) \neq U(\south{u})\}\) are equal.
	Thus, \eqref{eq:wtqt-prob-balanced-r} holds in this final case, as well.
\end{proof}

\begin{lemma} \label{lemma:wtqt-prob-balanced-h}
	Let \(T \in \NAF(\alpha, \sigma)\) be a non-attacking filling.
	Let \(U = \sswap[0, h]{i}(T)\) for some \(h \in \interval{0, \alpha_i}\).
	Then,
	\begin{equation} \label{eq:wtqt-prob-balanced-h}
		\wtqt[h](U)\left(1 - \pp[h]{i}(U)\right) = \wtqt[h](T)\left(1 - \pp[h]{i}(T)\right).
	\end{equation}
\end{lemma}

\begin{proof}
	If \(h = \alpha_i\),
	then \(\wtqt[\alpha_i](T) = \wtqt[\alpha_i](U) = 1\)
	and \(1 - \pp[\alpha_i]{i}(T) = 1 - \pp[\alpha_i]{i}(U) = 1\),
	so \eqref{eq:wtqt-prob-balanced-h} holds trivially.
	Assume that \(h < \alpha_i\).
	Then
	\begin{align*}
		c & = T(i, h+1)=U(i,h+1), &
		d & = T(i+1, h+1)=U(i+1,h+1), \\
		a & = T(i, h)=U(i+1,h), \text{ and } &
		b & = T(i+1, h)=U(i,h).
	\end{align*}
	Since \(U\) is non-attacking, we have \(b \neq d\).
	We note that
	\begin{equation*}
		\maj[h + 1](T) - \maj[h + 1](U) = \ind{c, a} + \ind{d, b} - \ind{c, b} - \ind{d, a}.
	\end{equation*}
	Moreover, we have
	\begin{align*}
		\coinv[h, h + 1](T) - \coinv[h, h + 1](U) & = \inv[h, h + 1](U) - \inv[h, h + 1](T)\\
		& = \ind{c, d, b} - \ind{c, d, a} +
		\sum_{\substack{u \in \Arm(i+1, h+1) \\ k = U(u)}}
		\left( \ind{c, k, b} + \ind{d,k, a} - \ind{c,k, a} - \ind{d,k, b} \right),
	\end{align*}
	which simplifies to
	\( (A+1) \left( \ind{c, a} + \ind{d, b} - \ind{c, b} - \ind{d, a} \right)\)
	using \cref{lemma:useful-ind-five-variables}, where \(A = \arm(i+1, h+1)\).

	We show that \eqref{eq:wtqt-prob-balanced-h} holds
	for all possible cases for \((a,b,c,d)\) with \(b\neq d\).

	First, suppose \(|\{a, b, c, d\}| = 4\).
	Then, the sets \(\{u \in \Arm(i+1, h+1) \suchthat T(u) \neq T(\south{u})\}\)
	and \(\{u \in \Arm(i+1, h+1) \suchthat U(u) \neq U(\south{u})\}\) are equal.
	\begin{itemize}[left=0pt]
		\item On the one hand, if \(\ind{c, d, a} = \ind{c, d, b}\),
			then \(1 - \pp[h]{i}(T) = 1\) by \cref{item:pp-cda-cdb}.
			Then, it also holds that \(\ind{d, c, b} = \ind{d, c, a}\),
			and hence \(1 - \pp[h]{i}(U) = 1\) by \cref{item:pp-cda-cdb}.
			It also holds that
			\begin{equation*}
				\ind{c, a} + \ind{d, b} - \ind{c, b} - \ind{d, a} = 0,
			\end{equation*}
			therefore, \(\maj[h + 1](T) = \maj[h + 1](U)\)
			and \(\coinv[h, h + 1](T) = \coinv[h, h + 1](U)\).
		\item On the other hand, if \(\ind{c, d, a}=1-\ind{c, d, b} = \ind{d, c, b}\),
			then \(1 - \pp[h]{i}(T) = 0\) by \cref{item:pp-cda-dcb}.
			Then, it also holds that \(\ind{d, c, b} = \ind{c, d, a}\),
			and hence \(1 - \pp[h]{i}(U) = 0\) by \cref{item:pp-cda-dcb}.
	\end{itemize}
	Thus, in both cases for \(\ind{c,d,a}\), \eqref{eq:wtqt-prob-balanced-h} holds for \(|\{a,b,c,d\}|=4\).

	Now suppose that \(|\{a, b, c, d\}| = 3\).
	Since \(b\neq d\), without loss of generality we may assume \(b=c\).
	By \cref{item:pp-bc}, we have
	\(1 - \pp[h]{i}(T) = 1\).
	Moreover,
	\begin{equation*}
		1 - \pp[h]{i}(U) = (q^{L-h}t^{A+1})^{1 - \ind{d, b, a}} \frac{1 - t}{1 - q^{L-h}t^{A+2}},
	\end{equation*}
	where \(A = \arm(i+1, h+1)\), by \cref{item:pp-ac}.
	It also holds that
	\begin{equation*}
		\ind{b, a} + \ind{d, b} - \ind{b, b} - \ind{d, a} = \ind{d, b, a}.
	\end{equation*}
	For the major index, we have \(\maj[r + 1](T) - \maj[r + 1](U)\) equals \( (\ell+1) \ind{d, b, a}\).
	For the coinversion number,
	we have \(\coinv[r, r + 1](T) - \coinv[r, r + 1](U)\)
	equals \((A + 1) \ind{d, b, a}\).
	Finally, we have
	\begin{equation*}
		\left(
			\prod_{\substack{u \in \dg[r + 1]{\alpha} \\ T(u) \neq T(\south{u})}}
			\frac{1 - t}{1 - q^{\leg(u)+1} t^{ \arm(u)+1}}
		\right)
		\big/
		\left(
			\prod_{\substack{u \in \dg[r + 1]{\alpha} \\ U(u) \neq U(\south{u})}}
			\frac{1 - t}{1 - q^{\leg(u)+1} t^{\arm(u)+1}}
		\right)
		= \frac{(1 - t)}{1 - q^{\ell+1}t^{A+1}},
	\end{equation*}
	since \(u=(i, r+1)\) is the only relevant box.
	Therefore, we compute
	\begin{equation*}
		\frac{\wtqt[r](T)}{\wtqt[r](U)} =
		q^{(\ell+1)\ind{d, b, a}} t^{(A+1)\ind{d, b, a}} \frac{(1 - t)}{1 - q^{\ell+1}t^{A+1}}
		= \pp[r]{i}(T, U),
	\end{equation*}
	and thus \eqref{eq:wtqt-prob-balanced-h} holds for this case as well.

	Finally, the case \(|\{a, b, c, d\}| = 2\) does not happen as \(b \neq d\)
	by the non-attacking property of \(U\).
	Therefore, \eqref{eq:wtqt-prob-balanced-h} holds in all cases.
\end{proof}

We are now ready to prove the main result of this section, \cref{theorem:wt-prob}.

\begin{proof}[Proof of \cref{theorem:wt-prob}]
	Assume that \(U = \sswap[0, h]{i}(T)\) for some \(h \in \interval{0, \alpha_i}\).
	If such \(h\) doesn't exist, then \(\prob{i}(T, U) = \prob{i}(U, T) = 0\), and the lemma holds trivially.
	For all \(r \in \interval{h + 1, \alpha_i-1}\),
	since the \(r\)\textsuperscript{th} and \((r+1)\)\textsuperscript{st} rows of \(T\) and \(U\) are equal,
	we have
	\begin{equation} \label{eq:wtqt-prob-balanced-r-large}
		\wtqt[r](T) = \wtqt[r](U)
	\end{equation}
	By \cref{lemma:wtqt-prob-balanced-r}, for all \(r \in \interval{0, h-1}\), we have
	\begin{equation} \label{eq:wtqt-prob-balanced-r-repeated}
		\wtqt[r](U)\pp[r]{i}(U) = \wtqt[r](T)\pp[r]{i}(T),
	\end{equation}
	and by \cref{lemma:wtqt-prob-balanced-h}, we have
	\begin{equation} \label{eq:wtqt-prob-balanced-h-repeated}
		\wtqt[h](U)\left(1 - \pp[h]{i}(U)\right) = \wtqt[h](T)\left(1 - \pp[h]{i}(T)\right).
	\end{equation}
	Multiplying
	\eqref{eq:wtqt-prob-balanced-r-large} for each \(r \in \interval{h + 1, \alpha_i-1}\),
	\eqref{eq:wtqt-prob-balanced-r-repeated} for all \(r \in \interval{0, h-1}\), and
	\eqref{eq:wtqt-prob-balanced-h-repeated},
	we obtain
	\begin{equation*}
		\left( \prod_{r = 0}^{\alpha_i-1} \wtqt[r](U) \right)
		\left( \prod_{r = 0}^{h-1} \pp[r]{i}(U) \right)
		\left( 1 - \pp[h]{i}(U) \right)
		=
		\left( \prod_{r = 0}^{\alpha_i-1} \wtqt[r](T) \right)
		\left( \prod_{r = 0}^{h-1} \pp[r]{i}(T) \right)
		\left( 1 - \pp[h]{i}(T) \right).
	\end{equation*}
	By \eqref{eq:rthcomponent}, and \cref{definition:prob}, we obtain
	\begin{equation*}
		\wtqt(U)\prob{i}(U, T) = \wtqt(T)\prob{i}(T, U). \qedhere
	\end{equation*}
\end{proof}

Recall that \(\NAF(\alpha, \sigma, \beta)\) denotes the set of non-attacking fillings of shape \(\alpha\) with content \(\beta\).
Since \(\prob{i}(T, U) = 0\) whenever \(\mathbf{x}^T \neq \mathbf{x}^U\),
we may restrict the map in \cref{theorem:wt-prob}
to obtain a probabilistic bijection between
\(\NAF(\alpha, \sigma, \beta)\) and \(\NAF(\alpha, \sigma s_i, \beta)\).

\begin{corollary} \label{cor:wt-prob}
	The pair of maps
	\begin{gather*}
		{\prob{i}} \colon \NAF(\alpha, \sigma, \beta) \times \NAF(\alpha, \sigma s_i) \to \rationals(q, t) \\
		{\prob{i}} \colon \NAF(\alpha, \sigma s_i, \beta) \times \NAF(\alpha, \sigma) \to \rationals(q, t)
	\end{gather*}
	defines a probabilistic bijection between
	\(\NAF(\alpha, \sigma, \beta)\) and \(\NAF(\alpha, \sigma s_i, \beta)\)
	with respect to the weight function \(\wtqt\).
\end{corollary}

\begin{example}
	Let \(n = 4\),
	\(\alpha = \composition{2, 2, 0, 1}\),
	\(\sigma = \window{3, 1, 2, 4}\),
	\(\beta = \composition{1, 0, 2, 2}\).
	Let \(i = 1\) and note that \(\alpha_1 = \alpha_2 = 2\).
	Moreover, note that \(\sigma s_i = \window{1, 3, 2, 4}\).
	\cref{cor:wt-prob} implies that the maps
	\( \prob{1} \colon \NAF(\alpha, \sigma, \beta) \times \NAF(\alpha, \sigma s_1, \beta) \to \rationals(q, t)\)
	and
	\( \prob{1} \colon \NAF(\alpha, \sigma s_1, \beta) \times \NAF(\alpha, \sigma, \beta) \to \rationals(q, t)\)
	form a probabilistic bijection.

	The set
	\(\NAF(\alpha, \sigma, \beta) = \NAF(\composition{2, 2, 0, 1}, \window{3, 1, 2, 4}, \composition{1, 0, 2, 2})\)
	consists of three fillings,
	\(\textcolor{red!50!black}{T_1}\),
	\(\textcolor{red!50!black}{T_2}\), and
	\(\textcolor{red!50!black}{T_3}\),
	and the set
	\(\NAF(\alpha, \sigma s_i, \beta) = \NAF(\composition{2, 2, 0, 1}, \window{1, 3, 2, 4}, \composition{1, 0, 2, 2})\) 
	consists of two fillings,
	\(\textcolor{blue!50!black}{U_1}\) and
	\(\textcolor{blue!50!black}{U_2}\),
	shown in the left and right sides of \cref{fig:prob-2201-3124-1022}, respectively.
	\cref{fig:prob-2201-3124-1022} shows the values of
	\(\prob{1}(T, U)\) and \(\prob{1}(U, T)\) for each pair of fillings.

	\begin{figure}[htbp]
		\begin{tikzpicture}
			\node[red!50!black]  (T1) at (-4,  3) {
				\(T_1 =
				\begin{ytableau}
					2 & 4 \\
					1 & 2 & \none & 4 \\
					3 & 1 & 2 & 4
				\end{ytableau}\)
			};
			\node[red!50!black]  (T2) at (-4,  0) {
				\(T_2 =
				\begin{ytableau}
					2 & 4 \\
					2 & 1 & \none & 4 \\
					3 & 1 & 2 & 4
				\end{ytableau}\)
			};
			\node[red!50!black]  (T3) at (-4, -3) {
				\(T_3 =
				\begin{ytableau}
					4 & 2 \\
					1 & 2 & \none & 4 \\
					3 & 1 & 2 & 4
				\end{ytableau}\)
			};
			\node[blue!50!black] (U1) at ( 4,  1.5) {
				\(\begin{ytableau}
					2 & 4 \\
					1 & 2 & \none & 4 \\
					1 & 3 & 2 & 4
				\end{ytableau} = U_1\)
			};
			\node[blue!50!black] (U2) at ( 4, -1.5) {
				\(\begin{ytableau}
					4 & 2 \\
					1 & 2 & \none & 4 \\
					1 & 3 & 2 & 4
				\end{ytableau} = U_2\)
			};
			\node[red!50!black, below] at (T1.south) {
				\(\scriptstyle \wtqt(T_1) =
				\frac{q^{4} {(1 - t)}^{4} t^{3}}{{(1 - q^{2} t^{3})} {(1 - q^{2} t^{2})} {(1 - q t^{2})} {(1 - q t)}}\)
			};
			\node[red!50!black, below] at (T2.south) {
				\(\scriptstyle \wtqt(T_2) =
				\frac{q {(1 - t)}^{2}}{{(1 - q^{2} t^{3})} {(1 - q t)}}\)
			};
			\node[red!50!black, below] at (T3.south) {
				\(\scriptstyle \wtqt(T_3) =
				\frac{q^{3} t^{2} {(1 - t)}^{3} }{{(1 - q^{2} t^{3})} {(1 - q^{2} t^{2})} {(1 - q t^{2})}}\)
			};
			\node[blue!50!black, below] at (U1.south) {
				\(\scriptstyle \wtqt(U_1) =
				\frac{q^{2} {(1 - t)}^{3} t}{{(1 - q^{2} t^{2})} {(1 - q t^{2})} {(1 - q t)}}\)
			};
			\node[blue!50!black, below] at (U2.south) {
				\(\scriptstyle \wtqt(U_2) =
				\frac{q {(1 - t)}^{2}}{{(1 - q^{2} t^{2})} {(1 - q t^{2})}}\)
			};
			\draw[-{Latex}, red!50!black] (T1) 
				edge[bend left=15] node[sloped, fill=white]{\(1\)} (U1);
			\draw[-{Latex}, blue!50!black] (U1) 
				edge[bend left=0] node[sloped, fill=white]{\(\tfrac{q^2t^2(1-t)}{1-q^2t^3}\)} (T1);
			\draw[-{Latex}, red!50!black] (T2) 
				edge[bend left=10] node[sloped, fill=white]{\(\frac{qt(1-t)}{1-qt^2}\)} (U1);
			\draw[-{Latex}, blue!50!black] (U1) 
				edge[bend left=6] node[sloped, fill=white]{\(\frac{1-q^2t^2}{1-q^2t^3}\)} (T2);
			\draw[-{Latex}, red!50!black] (T2) 
				edge[bend left=6] node[sloped, fill=white]{\(\frac{qt-1}{1-qt^2}\)} (U2);
			\draw[-{Latex}, blue!50!black] (U2) 
				edge[bend left=10] node[sloped, fill=white]{\(\frac{1-q^2t^2}{1-q^2t^3}\)} (T2);
			\draw[-{Latex}, red!50!black] (T3) 
				edge[bend left=0] node[sloped, fill=white]{\(1\)} (U2);
			\draw[-{Latex}, blue!50!black] (U2) 
				edge[bend left=15] node[sloped, fill=white]{\(\tfrac{q^2t^2(1-t)}{1-q^2t^3}\)} (T3);
		\end{tikzpicture}
		\caption{
			For each \(T \in \NAF(\composition{2, 2, 0, 1}, \window{3, 1, 2, 4}, \composition{1, 0, 2, 2})\)
			(right, in red),
			and each \(U \in \NAF(\composition{2, 2, 0, 1}, \window{1, 3, 2, 4}, \composition{1, 0, 2, 2})\)
			(left, in blue),
			the label of the red arrow from \(T\) to \(U\) is the value of \(\prob{1}(T, U)\),
			and the label of the blue arrow from \(U\) to \(T\) is the value of \(\prob{1}(U, T)\).
			The lack of an arrow indicates that the value is zero.
			Below each filling, we show the value of \(\wtqt\).
		}
		\label{fig:prob-2201-3124-1022}
	\end{figure}

	The fact that the maps form a probabilistic bijection can be checked
	by computing that, for each
	\(T \in \{\textcolor{red!50!black}{T_1}, \textcolor{red!50!black}{T_2}, \textcolor{red!50!black}{T_3}\}\)
	and each
	\(U \in \{\textcolor{blue!50!black}{U_1}, \textcolor{blue!50!black}{U_2}\}\),
	we have
	\begin{equation*}
		\prob{1}(T, U) \wtqt(T) = \prob{1}(U, T) \wtqt(U).
	\end{equation*}
	Moreover, using the fact that there exists a probabilistic bijection between
	\(\{\textcolor{red!50!black}{T_1}, \textcolor{red!50!black}{T_2}, \textcolor{red!50!black}{T_3}\}\)
	and
	\(\{\textcolor{blue!50!black}{U_1}, \textcolor{blue!50!black}{U_2}\}\),
	\cref{theorem:wt-prob} implies that their \(q, t\)-weight generating functions are equal,
	that is,
	\begin{equation*}
		\wtqt(\textcolor{red!50!black}{T_1})
		+ \wtqt(\textcolor{red!50!black}{T_2})
		+ \wtqt(\textcolor{red!50!black}{T_3})
		= \wtqt(\textcolor{blue!50!black}{U_1})
		+ \wtqt(\textcolor{blue!50!black}{U_2}).
	\end{equation*}
	This equality can be checked from the known values of \(\wtqt\) in \cref{fig:prob-2201-3124-1022}.
	Note that the left-hand side is the coefficient
	\([\mathbf{x}^{\composition{1, 0, 2, 2}}] E_{\composition{2, 2, 0, 1}}^{\window{3, 1, 2, 4}}(\mathbf{x}; q, t)\)
	and the right-hand side is the coefficient
	\([\mathbf{x}^{\composition{1, 0, 2, 2}}] E_{\composition{2, 2, 0, 1}}^{\window{1, 3, 2, 4}}(\mathbf{x}; q, t)\).
\end{example}

\subsection{Bijective proof of \texorpdfstring{\cref{theorem:main}}{symmetry theorem}}
\label{sec:proof}

We now have the tools to give a bijective proof for the main result of this article,
\cref{theorem:main}, a symmetry theorem for the permuted-basement Macdonald polynomials.
Let's recall the statement.

\main*

\begin{proof}
	We write
	\begin{equation*}
		E_{\alpha}^{\sigma}(\mathbf{x}; q, t)
		= \sum_{T \in \NAF(\alpha, \sigma)} \mkern-9mu \wt(T)
		= \sum_{\text{compositions }\beta} \mkern-9mu x^\beta
			\left(\sum_{T \in \NAF(\alpha, \sigma, \beta)} \mkern-9mu \wtqt(T)\right).
	\end{equation*}
	For each composition \(\beta = \composition{\beta_1, \ldots, \beta_n}\),
	\cref{cor:wt-prob} implies that there exists a probabilistic bijection between
	\(\NAF(\alpha, \sigma, \beta)\) and \(\NAF(\alpha, \sigma s_i, \beta)\)
	with respect to \(\wtqt\).
	Hence, \cref{proposition:equal-weight-sum} implies that
	\begin{equation*}
		[x^\beta] E_{\alpha}^{\sigma}(\mathbf{x}; q, t)
		= \sum_{T \in \NAF(\alpha, \sigma, \beta)} \mkern-9mu \wtqt(T)
		= \sum_{U \in \NAF(\alpha, \sigma s_i, \beta)} \mkern-9mu \wtqt(U)
		= [x^\beta] E_{\alpha}^{\sigma s_i}(\mathbf{x}; q, t),
	\end{equation*}
	and the result follows.
\end{proof}

Given a composition \(\alpha=\composition{\alpha_1,\ldots,\alpha_n}\),
let \(\mathscr{C}_\alpha\) be the subgroup of \(\sym{n}(\alpha)\)
generated by the transpositions \(\{s_i \suchthat \alpha_i = \alpha_{i+1}\}\).
Writing \(\alpha = p_1^{m_1} p_2^{m_2} \cdots p_k^{m_k}\),
we have that
\begin{equation*}
	\mathscr{C}_\alpha =
	\sym{\interval{1, m_1}} \times
	\sym{\interval{m_1+1, m_1+m_2}} \times
	\cdots \times
	\sym{\interval{m_1+m_2+\cdots+m_{k-1}+1, n}}
	\subseteq \sym{n}(\alpha).
\end{equation*}

For example, if \(\alpha =\composition{1, 1, 2, 2, 1, 3} = 1^2 2^2 1^1 3^1\),
then \(\mathscr{C}_\alpha\) is the subgroup of \(\sym{6}(\alpha)\) generated by
the transpositions \(s_1\) and \(s_3\):
\begin{equation*}
	\mathscr{C}_\alpha =
	\{
		\window{1, 2, 3, 4, 5, 6},
		\window{2, 1, 3, 4, 5, 6},
		\window{1, 2, 4, 3, 5, 6},
		\window{4, 1, 4, 3, 5, 6}
	\}.
\end{equation*}
The group \(\sym{6}(\alpha)\) is generated by \(\{s_1,s_3,s_{2,5}\}\),
where \(s_{2,5}=s_2s_3s_4s_3s_2\).
For instance, \(\window{5, 2, 3, 4, 1, 6}\in\sym{n}(\alpha)\setminus \mathscr{C}_\alpha\).

\begin{corollary}
	\label{corollary:C-alpha-symmetry}
	Let \(\alpha=\composition{\alpha_1,\ldots,\alpha_n}\) be a composition.
	Then, if \(\sigma, \tau \in \sym{n}\) are such that
	\(\sigma^{-1} \tau \in \mathscr{C}_\alpha\), then
	\begin{equation*}
		E_{\alpha}^{\sigma}(\mathbf{x}; q, t) = E_{\alpha}^{\tau}(\mathbf{x}; q, t).
	\end{equation*}
\end{corollary}

\begin{proof}
	By induction,
	it suffices to prove the statement for \(\sigma^{-1} \tau = s_i\)
	for some \(i\) such that \(\alpha_i = \alpha_{i+1}\).
	Then, the result follows from \cref{theorem:main}.
\end{proof}

If \(\alpha = p_1^{m_1} p_2^{m_2} \cdots p_k^{m_k}\),
then \(\mathscr{C}_\alpha = \sym{n}(\alpha)\)
if and only if \(p_1, p_2, \ldots, p_k\) are pairwise distinct.
In particular,
this is the case when \(\alpha\) is a partition or an antipartition.

\begin{corollary} \label{corollary:partition-antipartition-symmetry}
	Let \(\alpha = \composition{\alpha_1, \ldots, \alpha_n}\)
	be a partition or an antipartition.
	If \(\sigma, \tau \in \sym{n}\) are such that
	\(\sigma \permact \alpha = \tau \permact \alpha\),
	then
	\begin{equation*}
		E_{\alpha}^{\sigma}(\mathbf{x}; q, t) = E_{\alpha}^{\tau}(\mathbf{x}; q, t).
	\end{equation*}
\end{corollary}

\begin{proof}
	Since \(\sigma^{-1} \tau \in \sym{n}(\alpha)\),
	when \(\alpha\) is a partition or an antipartition,
	\(\sigma^{-1} \tau \in \mathscr{C}_\alpha\),
	and \cref{corollary:C-alpha-symmetry} applies.
\end{proof}

\section{Applications}\label{sec:applications}

Immediate applications of \cref{theorem:main} include the removal of certain assumptions from theorems in \cite{Ale19} and \cite{CMW22}.

\subsection{\texorpdfstring{\(t\)}{t}-atom polynomials and \texorpdfstring{\(t\)}{t}-key polynomials}
\label{sec:t-atoms}

Fix a composition \(\alpha = \composition{\alpha_1, \ldots, \alpha_n}\)
and a permutation \(\sigma \in \sym{n}\).
The \vocab{permuted basement \(t\)-atom polynomial of \(\alpha\) with basement \(\sigma\)},
first introduced in \cite{HLMvW11} and further studied in \cite{AS19}, is defined as
\begin{equation*}
	\atompoly_{\alpha}^{\sigma}(\mathbf{x}; t) = E_{\alpha}^{\sigma}(\mathbf{x}; 0, t).
\end{equation*}
Similarly, the \vocab{\(t\)-atom polynomial} \(\atompoly_{\alpha}(\mathbf{x}; t)\)
and the \vocab{\(t\)-key polynomial} \(\keypoly_{\alpha}(\mathbf{x}; t)\) are defined as
\begin{equation*}
	\atompoly_{\alpha}(\mathbf{x}; t) = \atompoly_{\alpha}^{\id}(\mathbf{x}; t)
	\qquad \text{and} \qquad
	\keypoly_{\alpha}(\mathbf{x}; t) = \atompoly_{\alpha}^{w_0}(\mathbf{x}; t).
\end{equation*}

\begin{proposition}[inferred from {\cite[Proposition~27]{Ale19}}]
	\label{proposition:atom-alpha-inc-alpha}
	Let \(\alpha = \composition{\alpha_1, \ldots, \alpha_n}\) be a composition.
	Let \(\sigma \in \sym{n}\) be the shortest permutation such that
	\(\sigma \permact \inc(\alpha) = \alpha\), and
	let \(\pi \in \sym{n}\) be the shortest permutation such that
	\(\pi \permact \dec(\alpha) = \alpha\).
	Then,
	\begin{equation*}
		\atompoly_{\alpha}(\mathbf{x}; t) = \atompoly_{\inc(\alpha)}^{\sigma}(\mathbf{x}; t)\qquad\text{and}\qquad
		\keypoly_{\alpha}(\mathbf{x}; t) = \atompoly_{\dec(\alpha)}^{\pi}(\mathbf{x}; t).
	\end{equation*}
\end{proposition}

Using \cref{corollary:partition-antipartition-symmetry},
we can remove the length restriction on \(\sigma\) and \(\pi\) above.

\begin{corollary} \label{lemma:atom-alpha-inc-alpha}
	Let \(\alpha = \composition{\alpha_1, \ldots, \alpha_n}\) be a composition.
	Let \(\tau \in \sym{n}\) be any permutation such that
	\(\tau \permact \inc(\alpha) = \alpha\), and
	let \(\pi \in \sym{n}\) be any permutation such that
	\(\pi \permact \dec(\alpha) = \alpha\).
	Then,
	\begin{equation*}
		\atompoly_{\alpha}(\mathbf{x}; t) = \atompoly_{\inc(\alpha)}^{\tau}(\mathbf{x}; t)
		\qquad\text{and}\qquad
		\keypoly_{\alpha}(\mathbf{x}; t) = \atompoly_{\dec(\alpha)}^{\pi}(\mathbf{x}; t).
	\end{equation*}
\end{corollary}

\begin{proof}
	Let \(\sigma \in \sym{n}\) be the shortest permutation such that \(\sigma \permact \inc(\alpha) = \alpha\).
	By Proposition \ref{proposition:atom-alpha-inc-alpha}, we have
	\(
		\atompoly_{\alpha}(\mathbf{x}; t)
		= \atompoly_{\inc(\alpha)}^{\sigma}(\mathbf{x}; t)
		= E_{\inc(\alpha)}^{\sigma}(\mathbf{x}; 0, t)
	\).
	Since \(\tau \permact \inc(\alpha) = \sigma \permact \inc(\alpha)\),
	\cref{corollary:partition-antipartition-symmetry} implies
	\(
		E_{\inc(\alpha)}^{\sigma}(\mathbf{x}; 0, t)
		= E_{\inc(\alpha)}^{\tau}(\mathbf{x}; 0, t)
		= \atompoly_{\inc(\alpha)}^{\tau}(\mathbf{x}; t)
	\).

	The proof for \(\keypoly_{\alpha}\) is analogous, and is omitted.
\end{proof}

\subsection{ASEP polynomials}\label{sec:mlqs}

In \cite{CMW22}, the authors introduce the ASEP polynomial \(F_\alpha(\mathbf{x}; q, t)\),
which are weight-generating functions for multiline queues with bottom row \(\alpha\).
A key result from their work is the following proposition.

\begin{proposition}[{\cite[Proposition~3.1]{CMW22}}] \label{proposition:multiline-queues}
	Let \(\alpha = \composition{\alpha_1, \ldots, \alpha_n}\) be a composition.
	Let \(\sigma \in \sym{n}\) be the longest permutation such that
	\(\sigma^{-1} \permact \alpha = \inc(\alpha)\).
	Then,
	\begin{equation*}
		F_{\alpha}(\mathbf{x}; q, t) = E_{\inc(\alpha)}^{\sigma}(\mathbf{x}; q, t).
	\end{equation*}
\end{proposition}

By \cref{corollary:partition-antipartition-symmetry},
we can remove the length requirement on \(\sigma\).

\begin{corollary}
	\label{lemma:multiline-queues}
	Let \(\alpha=\composition{\alpha_1,\ldots,\alpha_n}\) be a composition, and let
	\(\tau \in \sym{n}\) such that \(\tau^{-1} \permact \alpha = \inc(\alpha)\).
	Then,
	\begin{equation*}
		F_{\alpha}(\mathbf{x}; q, t) = E_{\inc(\alpha)}^{\tau}(\mathbf{x}; q, t).
	\end{equation*}
\end{corollary}

This allows us to realize \(t\)-atom polynomials as specializations of ASEP polynomials.

\begin{corollary}
	Let \(\alpha = \composition{\alpha_1, \ldots, \alpha_n}\) be a composition.
	Then,
	\begin{equation*}
		F_{\alpha}(\mathbf{x}; 0, t) = \atompoly_{\alpha}(\mathbf{x}; t).
	\end{equation*}
\end{corollary}

\begin{proof}
	Let \(\tau \in \sym{n}\) such that \(\tau^{-1} \permact \alpha = \inc(\alpha)\).
	From \cref{lemma:multiline-queues}
	and \cref{lemma:atom-alpha-inc-alpha}, we have
	\(
		F_{\alpha}(\mathbf{x}; 0, t)
		= E_{\inc(\alpha)}^{\tau}(\mathbf{x}; 0, t)
		= \atompoly_{\inc(\alpha)}^{\tau}(\mathbf{x}; t)
		= \atompoly_{\alpha}(\mathbf{x}; t)
	\).
\end{proof}

Finally, we rewrite \eqref{eq:P in Es} as follows.
\begin{corollary}
	Let \(\lambda=(\lambda_1,\ldots,\lambda_n)\) be a partition.
	The symmetric Macdonald polynomial satisfies
	\begin{equation*}
		P_\lambda(\mathbf{x};q,t)=\sum_{\mu\in S_n\cdot\lambda}E_{\inc(\lambda)}^{\sigma_\mu}(\mathbf{x};q,t),
	\end{equation*}
	where \(\sigma_\mu\in S_n\) is any permutation such that \(\sigma_\mu\permact \inc(\lambda)=\mu\).
\end{corollary}

\printbibliography

@article{Ale19,
	title        = {Non-symmetric {M}acdonald polynomials and {D}emazure--Lusztig operators},
	author       = {Alexandersson, Per},
	year         = {2019},
	volume       = {76},
	issn         = {1286-4889},
	pages        = {Art. B76d, 27},
	journaltitle = {Séminaire Lotharingien de Combinatoire},
	shortjournal = {Sém. Lothar. Combin.},
	url          = {https://www.mat.univie.ac.at/~slc/wpapers/s76alexand.html}
}

@misc{Man24,
	title        = {A compact formula for the symmetric {M}acdonald polynomials},
	author       = {Mandelshtam, Olya},
	year         = {2024},
	eprinttype   = {arxiv},
	eprint       = {2401.17223},
	primaryclass = {math.CO}
}

@thesis{Fer11,
	type        = {phdthesis},
	title       = {Row-strict quasisymmetric {S}chur functions, characterizations of {D}emazure atoms, and permuted basement nonsymmetric {M}acdonald polynomials},
	pagetotal   = {90},
	institution = {University of California, Davis},
	location    = {California, United States},
	author      = {Ferreira, Jeffrey Paul},
	year        = {2011},
	url         = {https://www.proquest.com/docview/940887941}
}

@article{CMW22,
	title        = {From multiline queues to {M}acdonald polynomials via the exclusion process},
	volume       = {144},
	issn         = {0002-9327},
	doi          = {10.1353/ajm.2022.0007},
	pages        = {395--436},
	number       = {2},
	journaltitle = {American Journal of Mathematics},
	shortjournal = {Am. J. Math.},
	author       = {Corteel, Sylvie and Mandelshtam, Olya and Williams, Lauren},
	year         = {2022}
}

@article{HHL08,
	title        = {A combinatorial formula for nonsymmetric {M}acdonald polynomials},
	volume       = {130},
	issn         = {1080-6377},
	url          = {https://muse.jhu.edu/pub/1/article/235639},
	pages        = {359--383},
	number       = {2},
	journaltitle = {American Journal of Mathematics},
	shortjournal = {Amer. J. Math.},
	author       = {Haglund, Jim and Haiman, Mark and Loehr, Nick},
	year         = {2008},
	publisher    = {Johns Hopkins University Press}
}

@article{CHMMW22,
	title        = {Compact formulas for {M}acdonald polynomials and quasisymmetric {M}acdonald polynomials},
	author       = {Corteel, Sylvie and Haglund, Jim and Mandelshtam, Olya and Mason, Sarah and Williams, Lauren},
	volume       = {28},
	issn         = {1420-9020},
	doi          = {10.1007/s00029-021-00721-7},
	pages        = {32},
	number       = {2},
	journaltitle = {Selecta Mathematica},
	shortjournal = {Sel. Math. New Ser.},
	year         = {2022}
}

@article{AF22,
	title        = {$q${RS}$t$: A probabilistic {R}obinson--{S}chensted correspondence for {M}acdonald polynomials},
	volume       = {2022},
	issn         = {1073-7928, 1687-0247},
	doi          = {10.1093/imrn/rnab083},
	pages        = {13505--13568},
	number       = {17},
	journaltitle = {International Mathematics Research Notices},
	author       = {Frieden, Gabriel and Schreier-Aigner, Florian},
	year         = {2022}
}

@article{BP19,
	title        = {{Y}ang--{B}axter field for spin {H}all--{L}ittlewood symmetric functions},
	volume       = {7},
	issn         = {2050-5094},
	url          = {https://www.cambridge.org/core/journals/forum-of-mathematics-sigma/article/yangbaxter-field-for-spin-halllittlewood-symmetric-functions/E9B9313A0C7A73BBC09097E3715B5FF2},
	doi          = {10.1017/fms.2019.36},
	pages        = {e39},
	journaltitle = {Forum of Mathematics, Sigma},
	author       = {Bufetov, Alexey and Petrov, Leonid},
	year         = {2019}
}

@article{GR22,
	title        = {Comparing formulas for type ${GL}_n$ {M}acdonald polynomials},
	author       = {Guo, Weiying and Ram, Arun},
	year         = {2022},
	journaltitle = {Algebraic Combinatorics},
	shortjournal = {Algebr. Comb.},
	number       = {5},
	volume       = {5},
	pages        = {849--883},
	issn         = {2589-5486},
	doi          = {10.5802/alco.227}
}

@article{Che95,
	title        = {Nonsymmetric {M}acdonald polynomials},
	volume       = {1995},
	issn         = {1073-7928},
	url          = {https://doi.org/10.1155/S1073792895000341},
	doi          = {10.1155/S1073792895000341},
	pages        = {483--515},
	number       = {10},
	journaltitle = {International Mathematics Research Notices},
	shortjournal = {International Mathematics Research Notices},
	author       = {Cherednik, Ivan},
	year         = {1995}
}

@article{Opd95,
	title        = {Harmonic analysis for certain representations of graded {H}ecke algebras},
	volume       = {175},
	issn         = {0001-5962},
	url          = {http://projecteuclid.org/euclid.acta/1485890867},
	doi          = {10.1007/BF02392487},
	pages        = {75--121},
	number       = {1},
	journaltitle = {Acta Mathematica},
	shortjournal = {Acta Math.},
	author       = {Opdam, Eric M.},
	year         = {1995}
}

@article{Mac96,
	author    = {Macdonald, Ian G.},
	title     = {Affine {Hecke} algebras and orthogonal polynomials},
	journal   = {As\-t\'e\-ris\-que},
	note      = {S\'eminaire Bourbaki, vol. 1994/95, expos\'es 790--804},
	pages     = {talk no.~797, pp.~189--207},
	publisher = {Soci\'et\'e math\'ematique de France},
	number    = {237},
	year      = {1996},
	url       = {http://www.numdam.org/item/SB_1994-1995__37__189_0/}
}

@article{HLMvW11,
	title        = {Quasisymmetric {S}chur functions},
	volume       = {118},
	issn         = {0097-3165},
	doi          = {10.1016/j.jcta.2009.11.002},
	pages        = {463--490},
	number       = {2},
	journaltitle = {Journal of Combinatorial Theory, Series A},
	shortjournal = {Journal of Combinatorial Theory, Series A},
	author       = {Haglund, Jim and Luoto, Kurt and Mason, Sarah and van Willigenburg, Stephanie},
	year         = {2011}
}

@article{AS19,
	title        = {Properties of non-symmetric {M}acdonald polynomials at $q=1$ and $q=0$},
	volume       = {23},
	issn         = {0219-3094},
	doi          = {10.1007/s00026-019-00432-z},
	pages        = {219--239},
	number       = {2},
	journaltitle = {Annals of Combinatorics},
	shortjournal = {Ann. Comb.},
	author       = {Alexandersson, Per and Sawhney, Mehtaab},
	year         = {2019}
}

@article{MR04,
	title        = {Parametrizations of flag varieties},
	volume       = {8},
	issn         = {1088-4165},
	doi          = {10.1090/S1088-4165-04-00230-4},
	pages        = {212--242},
	number       = {9},
	journaltitle = {Representation Theory of the American Mathematical Society},
	shortjournal = {Represent. Theory},
	author       = {Marsh, Bethany R. and Rietsch, Konstanze},
	year         = {2004}
}

@article{Kno97,
	author    = {Knop, Friederich},
	doi       = {10.1515/crll.1997.482.177},
	issn      = {1435-5345},
	number    = {482},
	pages     = {177--190},
	title     = {Integrality of two variable Kostka functions},
	journal   = {Journal für die reine und angewandte Mathematik},
	number    = {482},
	pages     = {177--190},
	publisher = {De Gruyter},
	year      = {1997}
}

\end{document}